\documentclass[12pt]{article}
\usepackage{amsfonts, amstext, amsmath, amsthm, amscd, amssymb, hyperref, float, longtable, fullpage}
\usepackage{epsfig, graphics, psfrag}
\usepackage{graphicx}
\usepackage{color}
\usepackage[font=small, labelfont=bf]{caption}
\usepackage{float}
\newtheorem{proposition}{Proposition}[section]
\newtheorem{theorem}{Theorem}[section]
\newtheorem{lemma}{Lemma}[section]

\newtheorem{definition}{Definition}[section] 
\let\olddefinition\definition
\renewcommand{\definition}{\olddefinition\normalfont}
\newtheorem{remark}{Remark}[section]
\let\oldremark\remark
\renewcommand{\remark}{\oldremark\normalfont}
\newtheorem{question}{Question}[section]
\let\oldquestion\question
\renewcommand{\question}{\oldquestion\normalfont}
\newtheorem{notation}{Notation}[section]
\let\oldnotation\notation
\renewcommand{\notation}{\oldnotation\normalfont}
\newtheorem{example}{Example}[section]
\let\oldexample\example
\renewcommand{\example}{\oldexample\normalfont}
\newtheorem{algorithm}{Algorithm}[section]

\newtheorem{method}{Method}[section]
\let\oldmethod\method
\renewcommand{\method}{\oldmethod\normalfont}
\newtheorem{strategy}{Strategy}[section]
\let\oldstrategy\strategy
\renewcommand{\strategy}{\oldstrategy\normalfont}

\title{The Linking-Unlinking Game}

\author{Adam Giambrone\\ \small{Elmira College}\\ \small{Elmira, NY} \and Jake Murphy\\ \small{Louisiana State University}\\ \small{Baton Rouge, LA}}

\date{\today}

\begin{document}

\maketitle

\begin{abstract}
Combinatorial two-player games have recently been applied to knot theory. Examples of this include the Knotting-Unknotting Game and the Region Unknotting Game, both of which are played on knot shadows. These are turn-based games played by two players, where each player has a separate goal to achieve in order to win the game. In this paper, we introduce the Linking-Unlinking Game which is played on two-component link shadows. We then present winning strategies for the Linking-Unlinking Game played on all shadows of two-component rational tangle closures and played on a large family of general two-component link shadows. 
\end{abstract}

\tableofcontents 

\section{Introduction}

Recently, a number of researchers have applied game theory to knot theory in the form of combinatorial games played on knot diagrams. Examples of such games include Twist Untangle (\cite{twist1}), the Knotting-Unknotting Game (\cite{KU}, \cite{Sums}), and the Region Unknotting Game (\cite{RU}). The game Twist Untangle is played between two people on a nontrivial diagram of the unknot formed by iteratively twisting the unknotted circle. Players take turns using either an R1 move or an R2 move (see Figure~\ref{RMoves}) to decrease the number of crossings and simplify the diagram. The winner is the player that reduces the diagram to the unknotted circle. For more details, see \cite{twist1}. The Knotting-Unknotting Game is played on a \textbf{shadow} of a knot (see Definition~\ref{linkshadow}). Players take turns \textbf{resolving} a crossing (see Definition~\ref{resolvedef}) until all crossings are resolved and a knot diagram is formed. One player, the Unknotter, wins if the resulting knot diagram is unknotted (represents the trivial knot) while the other player, the Knotter, wins if the resulting knot diagram is knotted (represents a nontrivial knot). For more details, see \cite{KU} and \cite{Sums}. The Region Unknotting Game is similar to the Knotting-Unknotting Game in that the goals of the Knotter and Unknotter remain the same. The key difference is that play now consists of resolving the set of crossings that are incident to a face of the knot shadow (or changing the crossing type of a crossing if it has already been resolved). For more details, see \cite{RU}.

In this paper, we will adapt the Knotting-Unknotting Game to be played on two-component link shadows. We call this new game the \textbf{Linking-Unlinking Game}. Here, players still take turns resolving crossings. One player, the \textbf{Unlinker}, wins if the resulting two-component link diagram is \textbf{splittable} (is equivalent to a two-component link diagram where the components are separated from each other) while the other player, the \textbf{Linker}, wins if the resulting two-component link diagram is \textbf{unsplittable} (is not splittable). 

Our main goal in this paper is to find winning strategies for playing the Linking-Unlinking Game on families of two-component link shadows. The first family of link shadows we explore are the shadows of the two-component links that arise as a closure of the rational tangle $(a_1, \ldots, a_n)$. Our first three main results are combined into the following theorem. 

\begin{theorem} \label{IntroThm1}
Suppose we have a shadow of a rational two-component link coming from a closure of the rational tangle $(a_1,\ldots,a_n)$. 
\begin{enumerate}
\item[(1)] If $a_{2k+1} = 0$ for all $k$, then the Unlinker wins. 
\item[(2)] If either 
\begin{enumerate}
\item[(a)] $a_{2k+1}\neq 0$ for at least one $k$ and all of the $a_i$ are even, 
\item[(b)] $n=2$ and both $a_1$ and $a_2$ are odd, or
\item[(c)] $n\geq3$, both $a_1$ and $a_n$ are odd, and all other $a_i$ are even, 
\end{enumerate}
then the second player has a winning strategy (regardless of their role).
\end{enumerate}
\end{theorem}

To extend the above results to all shadows of two-component rational tangle closures, we utilize a decomposition of the syllables of the tangle word $(a_1,\ldots,a_n)$ into syllables consisting of self-intersections and strings of syllables consisting of non-self-intersections (see Definition~\ref{tangleNSI} and Proposition~\ref{NSIs}). By combining this decomposition with the proof of Theorem~\ref{IntroThm1}, we are able to prove the following result. 

\begin{theorem} \label{IntroThm2}
Suppose we have a shadow of a rational two-component link coming from a closure of a rational tangle. 
\begin{enumerate}
\item[(1)] If the tangle word contains an even number of self-intersections, then the second player has a winning strategy (regardless of their role).
\item[(2)] If the tangle word contains an odd number of self-intersections, then the first player has a winning strategy (regardless of their role).
\end{enumerate} 
\end{theorem}

Finally, we conclude the paper by expanding our focus to general two-component link shadows, using a decomposition of the crossings of a link shadow into self-intersections and non-self-intersections (see Definition~\ref{self-intersection}) and using linking number arguments to prove the following result. 

\begin{theorem} \label{IntroThm3}
Suppose we have a shadow of a general two-component link. 
\begin{enumerate}
\item[(1)] If the shadow contains zero non-self-intersections, then the Unlinker wins. 
\item[(2)] If the shadow contains a nonzero number of non-self-intersections and an even number of self-intersections, then the Linker has a winning strategy when playing second.
\item[(3)] If the shadow contains a nonzero number of non-self-intersections and an odd number of self-intersections, then the Linker has a winning strategy when playing first.
\end{enumerate} 
\end{theorem}

The remainder of this paper is organized as follows. In Section~\ref{background}, we present background from knot theory, defining terminology and providing results that will be used later. Specifically, we begin by introducing knots and links, various types of knot and link projections, and notions of equivalence for these objects in Section~\ref{knots}. In Section~\ref{NSILN}, we define splittable and unsplittable link diagrams, self-intersections and non-self-intersections for link diagrams, and the linking number of a two-component link diagram. In Section~\ref{ratlink}, we define rational tangles and rational link diagrams. In Section~\ref{ratlinkprops}, we define self-intersections and non-self-intersections for rational tangles, determine exactly when a rational tangle will close to form a two-component link diagram, and provide a decomposition theorem for rational tangle words. 

In Section~\ref{LUgame}, we define the Linking-Unlinking Game and present winning strategies for playing the game on various two-component link shadows. Specifically, we define the Linking-Unlinking Game and present two key player strategies in Section~\ref{LUdef}, we present winning strategies for all shadows of two-component rational tangle closures in Section~\ref{stratsratlink}, and we present winning strategies for large families of general two-component link shadows in Section~\ref{GeneralLinks}.

\section{Definitions and Background} \label{background}

To begin, we introduce some basic ideas from knot theory. Many of the definitions that follow can be found in introductory knot theory textbooks such as \textit{An Interactive Introduction to Knot Theory} by Johnson and Henrich (\cite{KT}) and \textit{The Knot Book: An Elementary Introduction to the Mathematical Theory of Knots} by Colin Adams (\cite{KnotBook}).

\subsection{Knots, Knot Projections, and Knot Equivalence} \label{knots}

A mathematical knot is much like the everyday knot we tie in our shoelaces. The key difference is that mathematical knots form closed loops.

\begin{definition}\label{knot}
A \textbf{knot} is a piecewise-linear simple closed curve in three-dimensional space $\mathbb{R}^3$. Equivalently, a \textbf{knot} is a smooth embedding of a circle into $\mathbb{R}^3$. 
\end{definition}

Returning to our shoelace analogy, if we fused the loose ends of a shoelace together, we would have a model of a mathematical knot. Figure~\ref{exknot} provides some examples of knots. Note that, if we consider the piecewise-linear definition of a knot, we can assume that knots are made up of a very large number of line segments so that the strands of the knot appear smooth. We now shift our attention to links, which consist of knots.  

\begin{figure}
    \centering
    \includegraphics[width=5in]{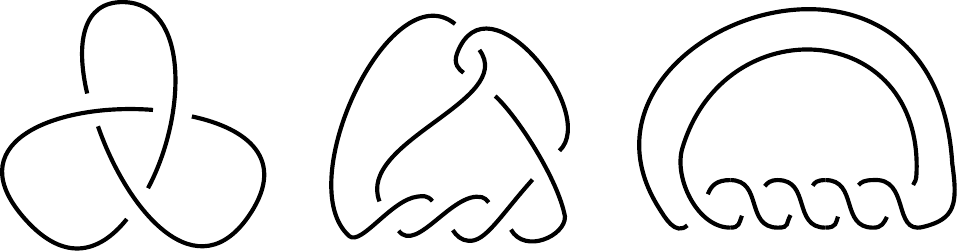}
    \caption{Some examples of knots.}
    \label{exknot}
\end{figure}

\begin{definition}
A \textbf{link} is a collection of one or more knots (that can be, but do not have to be, interlinked). An \textbf{n-component link} is a collection of $n$ knots.
\end{definition}

Note that a knot is a 1-component link, which means the study of links includes the study of knots. For the majority of this paper, we will focus on 2-component links. See Figure~\ref{exlink} for some examples of links with multiple components. To simplify the study of knotted loops in three dimensions, we will carefully project our links to two dimensions. 

\begin{figure}
    \centering
    \includegraphics[width=4.25in]{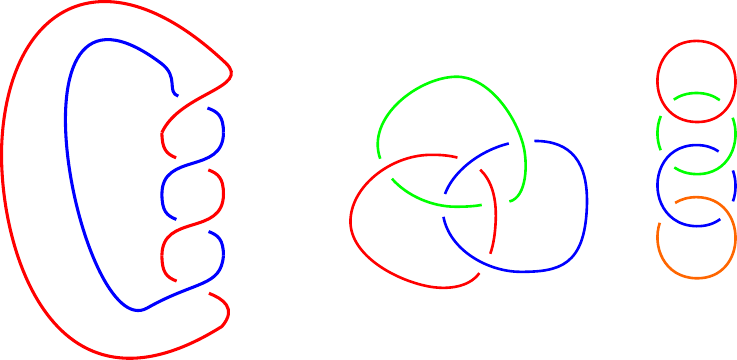}
    \caption{A 2-component link (left), a 3-component link (middle), and a 4-component link (right).}
    \label{exlink}
\end{figure}

\begin{definition}\label{linkshadow}
A \textbf{link shadow} is a projection of a link onto the plane $\mathbb{R}^2$ so that all crossings are transverse double crossings. 
\end{definition}

We can think of a link shadow as being formed by carefully shining a flashlight on a link and viewing the shadow it casts. Figure~\ref{shadow} provides both an example and a non-example of a link shadow.

\begin{figure}
    \centering
    \includegraphics{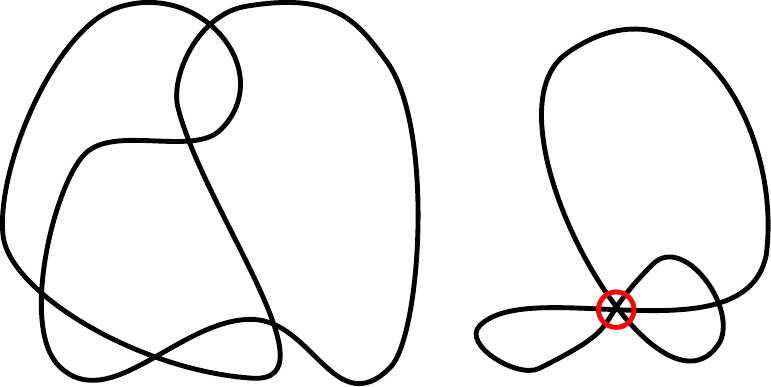}
    \caption{The figure on the left is a valid link shadow. The figure on the right is not, as the crossing circled in red has three strands meeting at the crossing, rather than two.}
    \label{shadow}
\end{figure}

Notice that link shadows only tell us where crossings occur, not which strand is above the other at each crossing. This means that multiple links can have the same link shadow. To indicate a specific link, we need to include more information.

\begin{definition} \label{resolvedef} 
A \textbf{resolved crossing} is a crossing in a link projection that has been \textbf{resolved} so that the \textbf{overstrand} and the \textbf{understrand} are distinguishable. This is usually depicted by adding two gaps to the understrand. A crossing that has not been resolved is called an \textbf{unresolved crossing}. 
\end{definition}

To resolve an unresolved crossing, we choose which strand becomes the overstrand. There are two possible resolutions for each crossing, as shown in Figure~\ref{resolve}. We now define the results of resolving a subset of crossings of a link shadow.  

\begin{figure}
    \centering
    \includegraphics[width=1.5in]{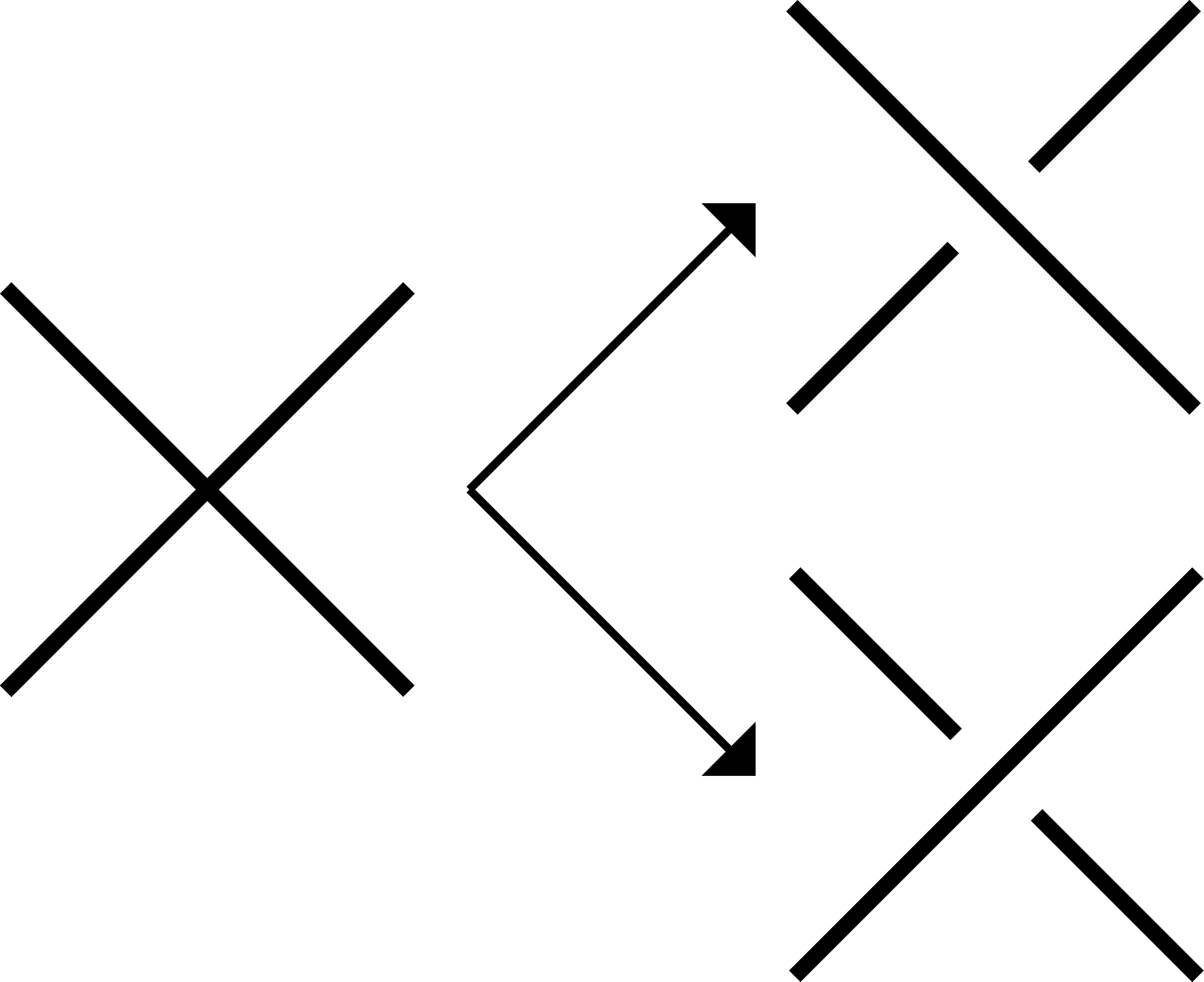}
    \caption{An unresolved crossing of a link projection (left) and the two possible resolutions for this crossing (right).}
    \label{resolve}
\end{figure}

\begin{definition}
A \textbf{link pseudodiagram} is a link projection where an arbitrary number of crossings have been resolved. From this perspective, a link shadow is a link pseudodiagram where no crossings have been resolved and a \textbf{link diagram} is a link pseudodiagram where every crossing has been resolved. 
\end{definition}

For an example of a link pseudodiagram, see Figure~\ref{ResolvedCrossings}. Given the family of all links and the family of all link diagrams, we will now discuss notions of equivalence for each of these families.

\begin{figure}
    \centering
    \includegraphics[width=2.5in]{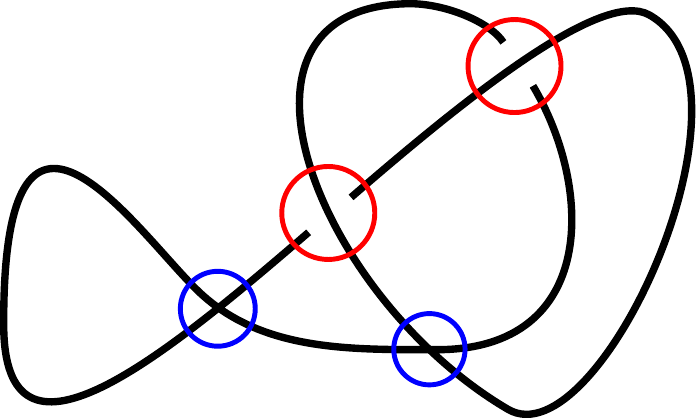}
    \caption{A link pseudodiagram where the crossings circled in red are resolved and the crossings circled in blue are unresolved.}
    \label{ResolvedCrossings}
\end{figure}

\begin{definition} If we can manipulate three-dimensional space $\mathbb{R}^3$ to deform one link $L_1$ into another link $L_2$, then $L_1$ and $L_2$ are called \textbf{equivalent}. 
\end{definition}

Link equivalence allows the collection of all links to be partitioned into equivalence classes. To determine equivalence on a diagrammatic level, we use \textbf{Reidemeister moves}, which are depicted in Figure~\ref{RMoves}, and \textbf{planar isotopies}, which are depicted in Figure~\ref{PlanarIsotopy}. Both Reidemeister moves and planar isotopies are local moves, meaning they occur within a fixed region of the plane (so the diagram outside of this region remains unchanged and is, therefore, ommitted from Figure~\ref{RMoves} and Figure~\ref{PlanarIsotopy}). A planar isotopy can be thought of as stretching and bending a single strand of a link diagram without affecting the crossing structure of the diagram. An example of a planar isotopy is shown in Figure~\ref{PIEx}. 

\begin{figure}
    \centering
    \includegraphics[width=2.5in]{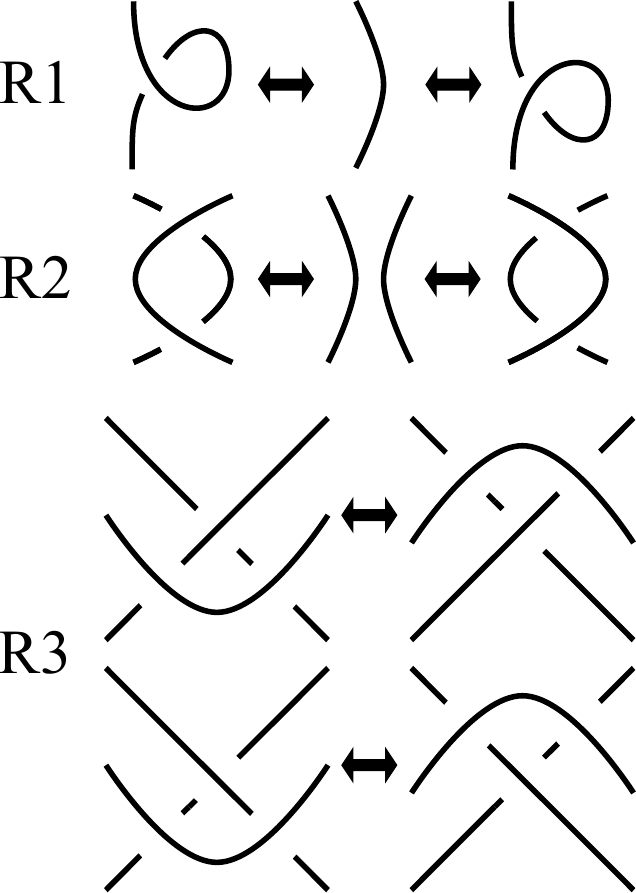}
    \caption{The three Reidemeister moves: R1 moves add or remove a loop, R2 moves overlay one strand on top of a nearby strand (or the reverse process), and R3 moves slide a strand over a crossing.}
    \label{RMoves}
\end{figure}

\begin{figure}
    \centering
    \includegraphics[width=2in]{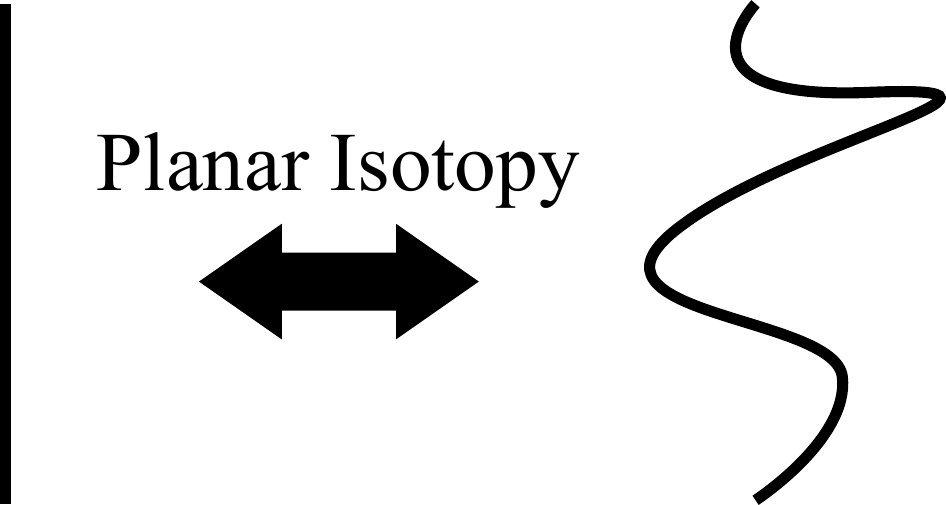}
    \caption{A general planar isotopy.}
    \label{PlanarIsotopy}
\end{figure}

\begin{figure}
    \centering
    \includegraphics{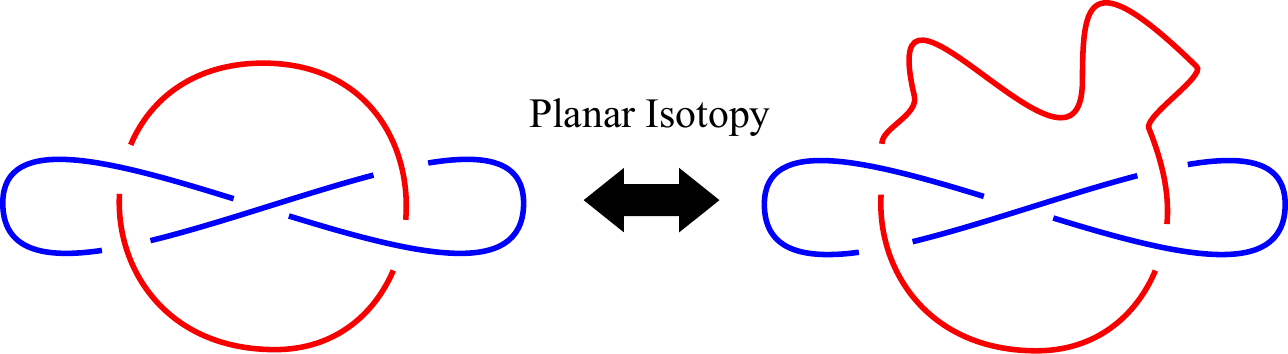}
    \caption{An example of a planar isotopy.}
    \label{PIEx}
\end{figure}

\begin{definition} If there is a finite sequence of Reidemeister moves and planar isotopies that turns a given link diagram $D_1$ into another link diagram $D_2$, then $D_1$ and $D_2$ are called \textbf{equivalent}. 
\end{definition}

As was the case with link equivalence, link diagram equivalence allows the collection of all link diagrams to be partitioned into equivalence classes. In this paper, we will focus almost exclusively on the R2 move.

\subsection{Splittable Link Diagrams, Non-Self-Intersections, and Linking Numbers} \label{NSILN}

The ability to determine whether or not the components of a 2-component link diagram can be separated from each other will be crucial for determining the winner of the Linking-Unlinking Game. As such, we divide the family of link diagrams into splittable link diagrams and unsplittable link diagrams. 

\begin{definition}
A link diagram is called \textbf{splittable} if it is equivalent to a \textbf{split link diagram} where one or more components can be drawn entirely within a circle, while the other components are entirely outside of the circle. If a link diagram is not splittable (resp. not split), then it is called \textbf{unsplittable} (resp. \textbf{non-split}).
\end{definition}

If a link diagram is splittable, we can imagine being able to split one or more components away from the rest of the link diagram. Figure~\ref{split} and Figure~\ref{unsplit} provide examples of splittable and unsplittable 2-component link diagrams, respectively. 

\begin{figure}
    \centering
    \includegraphics[width=2.5in]{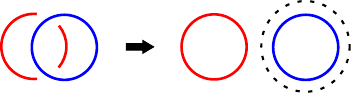}
    \caption{An example of a splittable link diagram (where each link diagram component in this case is an unknotted circle).}
    \label{split}
\end{figure}

\begin{figure}
    \centering
    \includegraphics[width=1.5in]{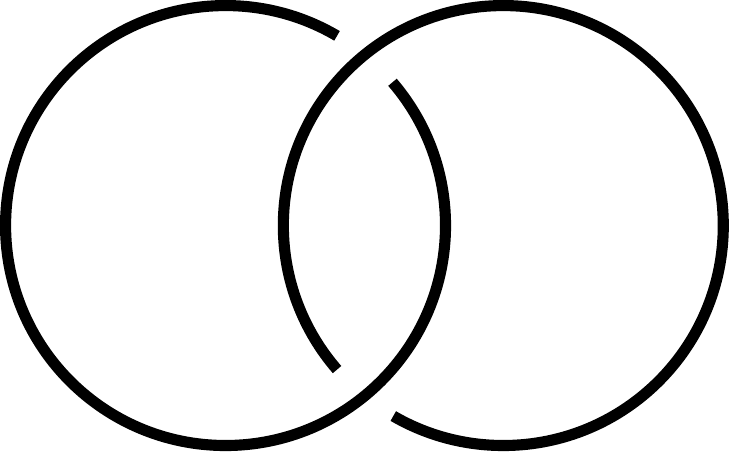}
    \caption{An example of an unsplittable link diagram called the Hopf link diagram.}
    \label{unsplit}
\end{figure}

Now with an understanding of splittable and unsplittable link diagrams, we will introduce an invariant of 2-component links called the linking number which will be used to help us detect when a link diagram is unsplittable. To define this quantity, we first need to discuss orientations of link pseudodiagrams. 
 
\begin{definition}
An \textbf{oriented} link pseudodiagram is a link pseudodiagram where one of two possible directions of travel has been chosen for each component of the link pseudodiagram.
\end{definition}

For an example of an oriented 2-component link diagram, see Figure~\ref{writhe}. Given an oriented link pseudodiagram, we can associate signs to each resolved crossing. 

\begin{figure}
    \centering
    \includegraphics[width=1.75in]{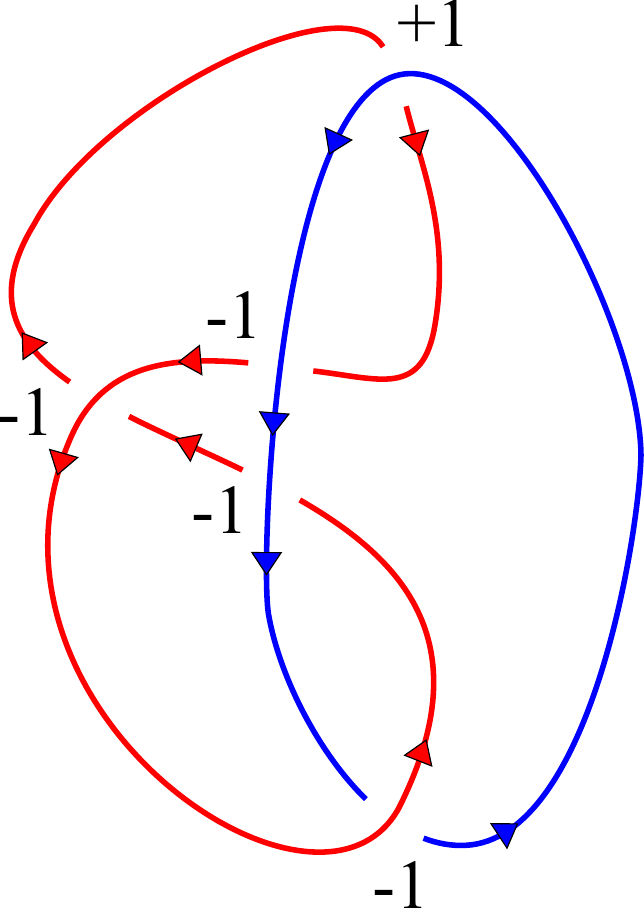}
    \caption{An oriented 2-component link diagram labeled with the signs of each of its crossings. Arrows are used to denote the choice of direction of travel.}
    \label{writhe}
\end{figure}

\begin{definition}
To each resolved crossing of an oriented link pseudodiagram, we can associate a \textbf{crossing sign} as shown in Figure~\ref{sign}. 
\end{definition}

\begin{figure}
    \centering
    \includegraphics[width=2in]{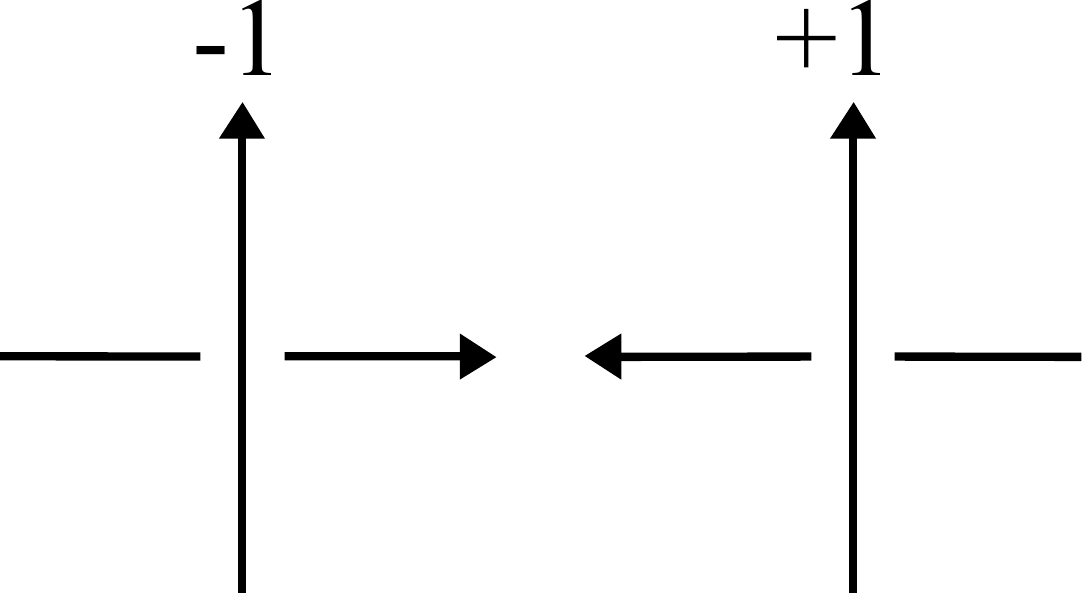}
    \caption{The left figure shows a $-1$ crossing and the right figure shows a $+1$ crossing.}
    \label{sign}
\end{figure}

In Figure~\ref{writhe}, each crossing is labeled with its sign. The final ingredient needed to define the linking number is the classification of the crossings of a link pseudodiagram into self-intersections and non-self-intersections.

\begin{definition} \label{self-intersection}
 If the two strands meeting at a crossing of a link pseudodiagram come from the same component of the link pseudodiagram, then such a crossing is called a \textbf{self-intersection (SI)}. If a crossing is not a self-intersection, then we call it a \textbf{non-self-intersection (NSI)}.
\end{definition}

Note that all crossings of a knot pseudodiagram are necessarily SIs. Consequently, NSIs can only occur in link pseudodiagrams containing at least two components. Figure~\ref{exsi} provides examples of SIs and NSIs in a link diagram. The idea of classifying the crossings of a link pseudodiagram as SIs or NSIs will be used in the definition of the linking number below as well as later on in this paper. 

\begin{figure}
    \centering
    \includegraphics[width=2in]{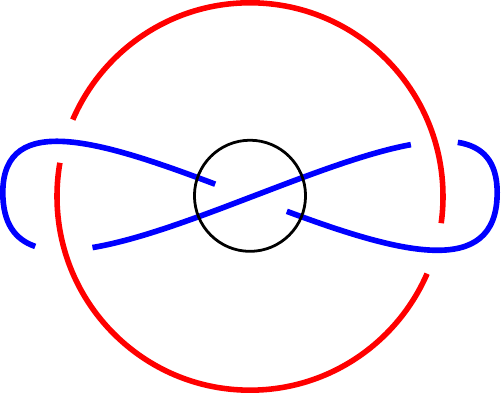}
    \caption{The circled crossing is a self-intersection. All other crossings are non-self-intersections.}
    \label{exsi}
\end{figure}

\begin{definition}
The \textbf{linking number} of an oriented 2-component link diagram is defined to be half of the sum of the crossing signs, where the sum is taken over all of the non-self-intersections (NSIs) of the link diagram.
\end{definition}

Looking back at Figure~\ref{writhe}, we see that the 2-component link diagram has a linking number of $\frac{1}{2}[3(-1)+1(1)]=-1$. Observe that the leftmost crossing, which has crossing sign $-1$, is not included in the linking number computation because this crossing is an SI. The linking number is an invariant of oriented 2-component link diagrams. This means that if two oriented 2-component link diagrams are equivalent, then they share the same linking number. Note that any splittable 2-component link diagram has linking number 0, as such a link diagram is equivalent to a link diagram with no NSIs. This tells us that any oriented 2-component link diagram with a nonzero linking number is unsplittable. However, an oriented 2-component link diagram with linking number 0 is not necessarily splittable. An example of this is shown in Figure~\ref{WH}. To compute the linking number in this example, we ignore the center crossing because it is an SI. The two crossings on the left have a negative sign while the two crossings on the right have a positive sign. Summing these signs gives a linking number of $\frac{1}{2}[2(-1)+2(1)]=0$, even though the link diagram is known to be unsplittable through other methods.

\begin{figure}
    \centering
    \includegraphics[width=2.5in]{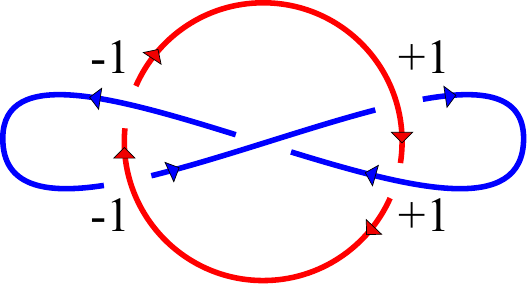}
    \caption{The Whitehead link is unsplittable with a linking number of $0$.}
    \label{WH}
\end{figure}

The following proposition, which will be useful later in this paper, determines the parity of the number of NSIs in a 2-component link pseudodiagram.

\begin{proposition} \label{evenNSI} 
Every 2-component link pseudodiagram contains an even number of non-self-intersections (NSIs).
\end{proposition}

\begin{proof}
The result is clearly true for a split 2-component link pseudodiagram, which contains no NSIs. Now consider a non-split 2-component link pseudodiagram, which necessarily contains NSIs. Choose a component of the link pseudodiagram and call it Component~1. We can distinguish between the ``inside'' and ``outside'' of Component~1 by giving it a canonical checkerboard coloring, that is, by coloring the regions of the shadow of Component~1 either black or white so that regions sharing an edge have opposite colors and so that the unbounded region is colored white. We then view the black (resp. white) regions as the ``inside'' (resp. ``outside'') of Component~1. An example of a canonical checkerboard coloring is shown in Figure~\ref{checker}. It is a classical result that every link pseudodiagram has a checkerboard coloring. (For more details, see Section~X.6 of \cite{Bollobas}.) 

\begin{figure}
    \centering
    \includegraphics[width=2in]{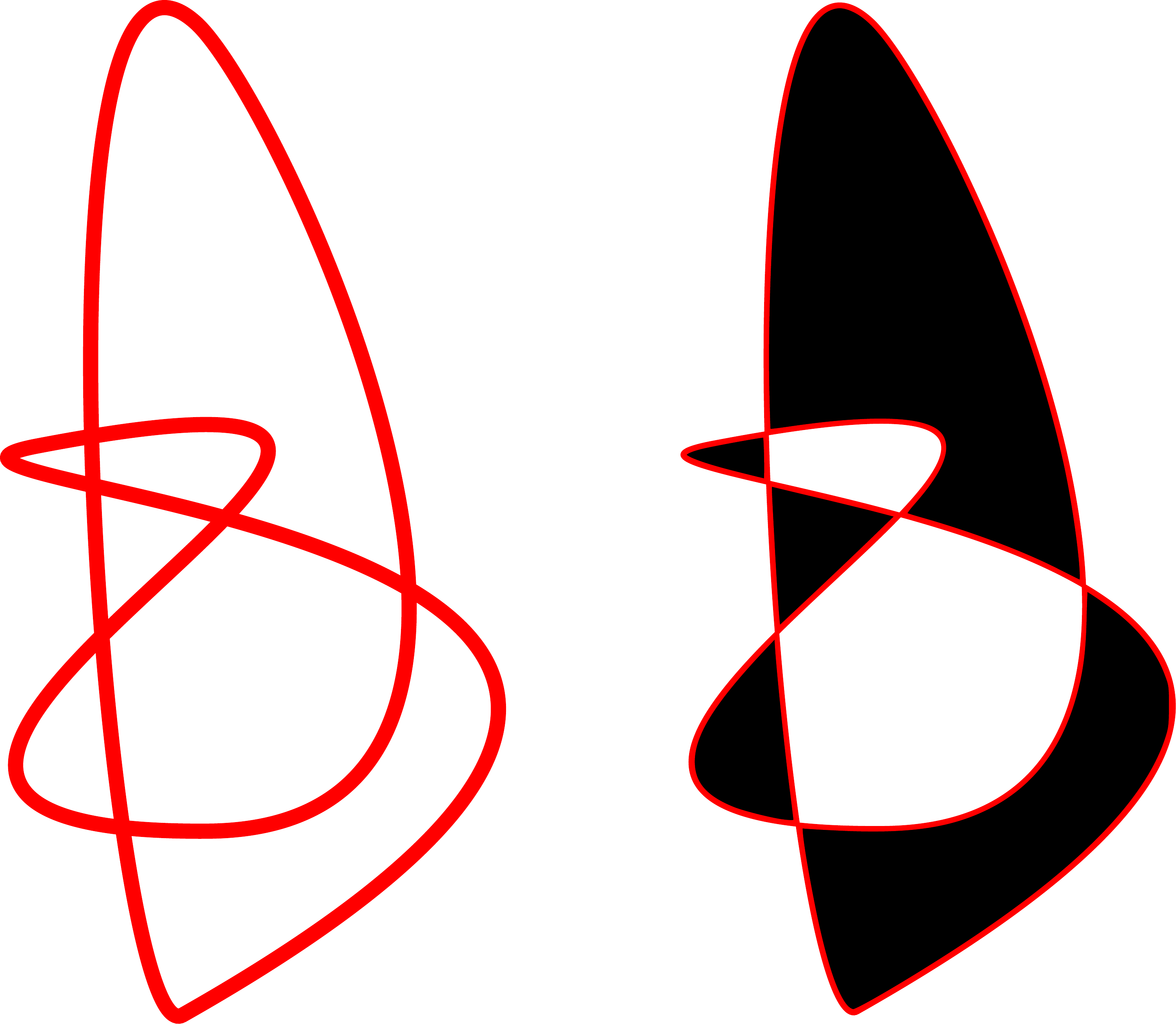}
    \caption{A knot shadow on the left with its canonical checkerboard coloring on the right. The black (resp. white) regions indicate the ``inside'' (resp. ``outside'') of the knot shadow.}
    \label{checker}
\end{figure}

Call the second component of the link pseudodiagram Component~2. Assign Component~2 an orientation and choose a starting point on Component~2 that is outside of Component~1. Since there are NSIs, at some point Component~2 will cross from the outside of Component~1 to the inside of Component~1. Since Component~2 enters the inside of Component~1, then it must also exit the inside of Component~1 because Component~2 is a closed curve that starts outside of Component~1. This follows from the Jordan Curve Theorem, which states that every simple closed curve in the plane has an interior and an exterior. (For more details, see Chapter~3 of \cite{Henle}.) Thus, the link pseudodiagram must contain at least two NSIs. 

Following the orientation of Component~2 from its starting point, the first NSI will bring us from the outside of Component~1 to the inside of Component~1. The second NSI will bring us back to the outside of Component~1. By iterating this argument, we can see that we are outside of Component~1 after passing through an even number of NSIs and we are inside of Component~1 after passing through an odd number of NSIs. 

If the link pseudodiagram contained an odd number of NSIs, then we would end up inside of Component~1 after beginning at the starting point, following the orientation of Component~2, and returning to the starting point. This means that the endpoint of Component~2 is inside of Component~1 and the starting point of Component~2 is outside of Component~1. This is a contradiction since the starting point and the endpoint of Component~2 are the same point. Therefore, there cannot be an odd number of NSIs in a 2-component link pseudodiagram.
\end{proof}

\subsection{Rational Tangles and Rational Link Diagrams} \label{ratlink}

Rational links come from rational tangles that are formed by an iterative process of twisting two strands. 
 
\begin{definition}
We define a \textbf{rational tangle}, denoted $(a_1,a_2,\ldots,a_n)$, where all of the $a_i$ are integers, through the following construction. We begin with two parallel vertical strands, as shown in Figure~\ref{RT0}. We then read $(a_1,a_2,\ldots,a_n)$ from left to right, twisting strands as we go. In particular, we begin by twisting the bottom two endpoints $|a_1|$ times and then proceed to alternate between twisting the right two endpoints $|a_{2k}|$ times and twisting the bottom two endpoints $|a_{2k+1}|$ times. If $a_i>0$ (resp. $a_{i}<0$), then we apply $|a_i|$ twists in such a way that the overstrand has a positive (resp. negative) slope. We call the $a_i$ the \textbf{syllables} of the \textbf{rational tangle word} $(a_1,a_2,\ldots,a_n)$. 
\end{definition}

\begin{figure}
\centering
\includegraphics[width=.5in]{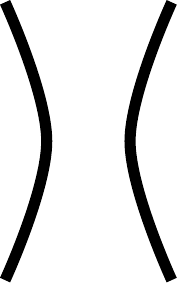}
\caption{The rational tangle (0), which consists of two vertical strands.}
\label{RT0}
\end{figure}

See Figure \ref{RT2321} for an example of the construction of a rational tangle. Since the Linking-Unlinking Game will be played on 2-component link pseudodiagrams, we need to generalize our notation for rational tangle words to allow for unresolved crossings.

\begin{definition}{(\cite{Sums})} By making a subset of the crossings of a rational tangle unresolved, we create a \textbf{rational pseudotangle}. We denote a rational pseudotangle by a \textbf{rational pseudotangle word} $(a_1(b_1),\ldots,a_n(b_n))$, where the $a_i(b_i)$ are called the \textbf{syllables} of the word and where $|a_i|$ denotes the number of resolved crossings in the $i^{\text{th}}$ syllable and $|b_i|$ denotes the number of unresolved crossings in the $i^{\text{th}}$ syllable. If either $a_i=0$ or $b_i=0$ (but not both) for some $i$, then we omit the single occurrence of $0$ or $(0)$ in the rational pseudotangle word. If both $a_i=0$ and $b_i=0$ for some $i$, then we replace $a_i(b_i)$ by $0$ in the rational pseudotangle word. 
\end{definition}

\begin{figure}
\centering
\includegraphics[width=6in]{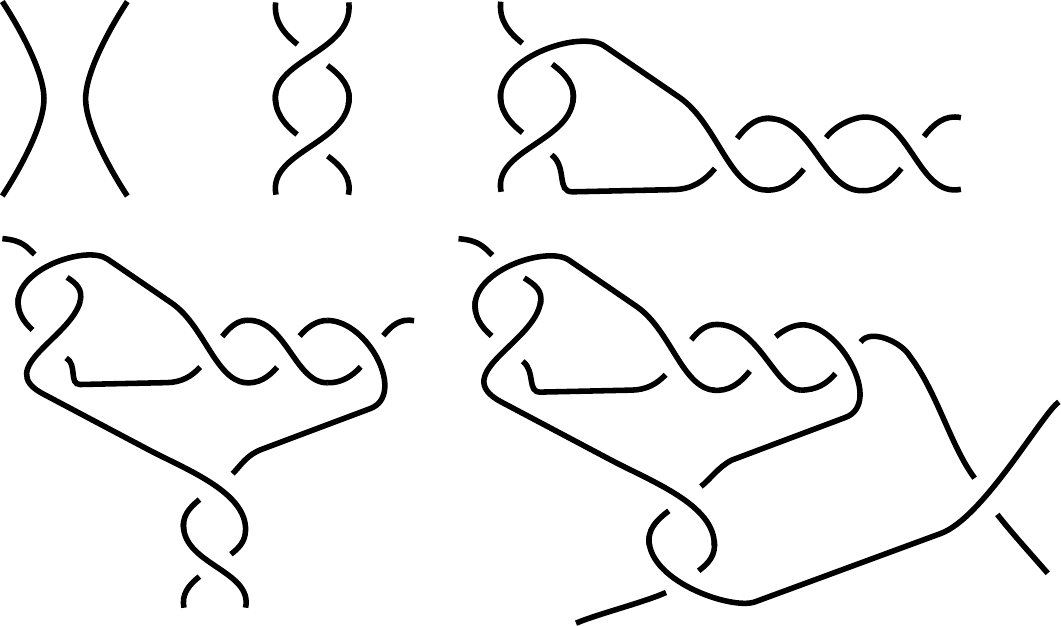}
\caption{The rational tangles $(0)$ (top left), $(2)$ (top center), $(2,-3)$ (top right), $(2,-3,-2)$ (bottom left), and $(2,-3,-2,1)$ (bottom right).}
\label{RT2321}
\end{figure}

We now present a number of tangle equivalences that will be useful later in this paper. 

\begin{proposition}[Lemma 4.2 from \cite{Sums}] \label{tangle eq}
The following statements provide a set of tangle equivalences for rational tangles. Similar statements can also be made for rational pseudotangles.
\begin{enumerate}
\item[0.] $\displaystyle (a_1,\ldots,a_i,0)=(a_1,\ldots,a_i)$
\item[1.] $\displaystyle (a_1,\ldots,a_i,0,a_{i+1},\ldots,a_n)=(a_1,\ldots,a_i+a_{i+1},\ldots,a_n)$  
\item[2.] $\displaystyle(a_1,\ldots,a_i,0,0,a_{i+1},\ldots,a_n)=(a_1,\ldots,a_i,a_{i+1},\ldots,a_n)$
\item[3.] $\displaystyle(0,a_1+1,a_2,\ldots,a_n)=(0,a_1,a_2,\ldots,a_n)$
\item[4.] $\displaystyle(1,a_1,a_2,\ldots,a_n)=(a_1+1,a_2,\ldots,a_n)$
\item[5.] $\displaystyle(-1,a_1,a_2,\ldots,a_n)=(a_1-1,a_2,\ldots,a_n)$
\end{enumerate}
\end{proposition}

Given that a rational pseudotangle has four endpoints, there are two ways to close a rational pseudotangle to form a rational link pseudodiagram.  

\begin{definition} \label{tangleclosure}
A rational pseudotangle can be closed to form a \textbf{rational link pseudodiagram} either by connecting the top endpoints together and the bottom endpoints together, forming the \textbf{numerator closure}, or by connecting the left endpoints together and the right endpoints together, forming the \textbf{denominator closure}. 
\end{definition}

The left side of Figure~\ref{TBLRC} shows the numerator closure of the rational tangle $(2,-3,-2,1)$ from Figure~\ref{RT2321}, which creates a knot diagram, and the right side of Figure~\ref{TBLRC} shows the denominator closure  of this tangle, which creates a 2-component link diagram. In general, either both closures of a rational pseudotangle will be knot pseudodiagrams or one closure will be a knot pseudodiagram and the other will be a 2-component link pseudodiagram. 

\begin{figure}
\centering
\includegraphics[width=5.5in]{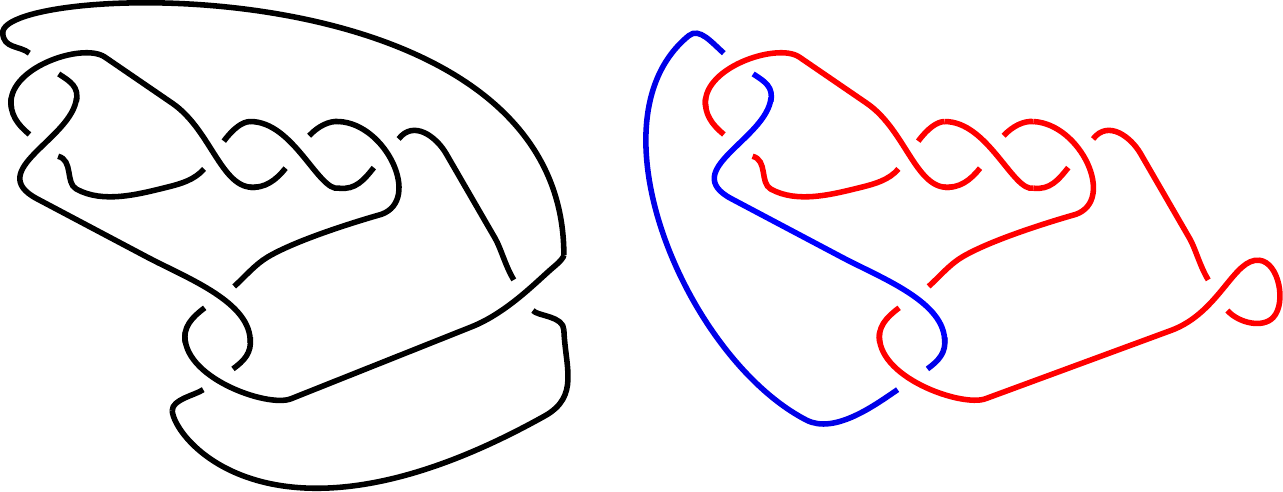}
\caption{The numerator closure of the tangle $(2,-3,-2,1)$ (above on the left) produces a knot diagram and the denominator closure of the tangle $(2,-3,-2,1)$ (above on the right) produces a 2-component link diagram.}
\label{TBLRC}
\end{figure}



\subsection{Non-Self-Intersections for Rational Tangles} \label{ratlinkprops} 

In Definition~\ref{self-intersection}, we defined self-intersections (SIs) and non-self-intersections (NSIs) for link pseudodiagrams. We will now define SIs and NSIs for rational pseudotangles.  

\begin{definition} \label{tangleNSI}
A crossing that occurs between the two strands of a rational pseudotangle is called a \textbf{self-intersection (SI)}. Otherwise, the crossing is called a \textbf{non-self-intersection (NSI)}. Furthermore, a syllable of a rational pseudotangle word is called an \textbf{SI syllable} (resp. \textbf{NSI syllable}) if all of the crossings in the syllable are SIs (resp. NSIs).  
\end{definition}

If a rational pseudotangle has a closure that is a 2-component link pseudodiagram, then the two strands of the pseudotangle become the two separate components of the link pseudodiagram and the SIs and NSIs of the pseudotangle become the SIs and NSIs of the link pseudodiagram, respectively. If a rational pseudotangle does not have a closure that is a 2-component link pseudodiagram (if both closures result in a knot pseudodiagram), then both the SIs and the NSIs of the pseudotangle become SIs of the knot pseudodiagram since knot pseudodiagrams cannot contain NSIs. 

When looking for winning strategies for the Linking-Unlinking Game played on rational 2-component link shadows, we need to make sure that there actually exists a closure of the rational pseudotangle that is a 2-component link pseudodiagram. The following result addresses this issue.

\begin{proposition}\label{ratevenNSI}
A rational pseudotangle has a closure that is a 2-component link pseudodiagram if and only if the pseudotangle contains an even number of NSIs.
\end{proposition}

\begin{proof}
($\Rightarrow$) Assume we have a 2-component link pseudodiagram that is a closure of a rational pseudotangle. By Proposition~\ref{evenNSI}, the link pseudodiagram must contain an even number of NSIs. This implies that the rational pseudotangle must also contain an even number of NSIs because, otherwise, the rational pseudotangle would contain an odd number of NSIs and closure would create a 2-component link pseudodiagram with an odd number of NSIs, contradicting Proposition~\ref{evenNSI}. 

\bigskip

\noindent ($\Leftarrow$) Assume we have a rational pseudotangle that contains an even number of NSIs and suppose, for a contradiction, that both the numerator and denominator closures result in a knot pseudodiagram. The only way this can happen is for one strand of the pseudotangle to have endpoints in the northeast and southwest corners and the other strand to have endpoints in the northwest and southeast corners, as shown on the left side of Figure~\ref{ratevencombined}.

\begin{figure}
\centering
\includegraphics[width=5in]{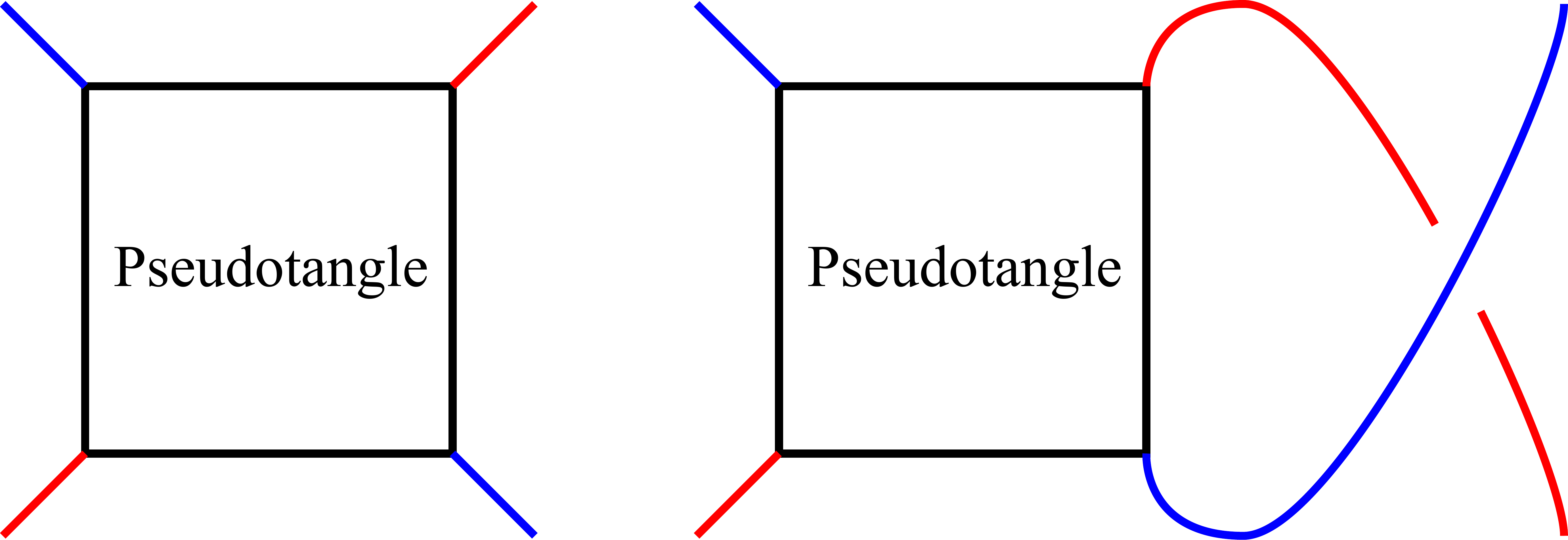}
\caption{A rational pseudotangle where the red strand has endpoints in the northeast and southwest corners and the blue strand has endpoints in the northwest and southeast corners (left) and the result of adding a half-twist to this pseudotangle (right).}
\label{ratevencombined}
\end{figure}


Add a half-twist to the end of the pseudotangle (which corresponds to the last syllable of the pseudotangle word), as shown on the right side of Figure~\ref{ratevencombined}. This produces a rational pseudotangle with an odd total number of NSIs. We can see that the numerator closure of this new rational pseudotangle will result in a 2-component link pseudodiagram. But this means we have a 2-component link pseudodiagram with an odd total number of NSIs, which contradicts Proposition~\ref{evenNSI}. 
\end{proof}



Proposition~\ref{SI} below provides information about SIs and NSIs for rational pseudotangle words. 

\begin{proposition} \label{SI} The following statements are true for rational tangle words $(a_1,\ldots,a_n)$. Similar statements can also be made for rational pseudotangle words.
\begin{itemize}
\item[(1)] The first syllable $a_1$ is an NSI syllable. 
\bigskip
\item[(2)] The second syllable $a_2$ is an SI syllable if and only if $a_1$ is even.
\bigskip
\item[(3)] If $a_i$ is an SI syllable, then $a_{i+1}$ is an NSI syllable.
\bigskip
\item[(4)] If $a_i$ is an SI syllable, then $a_{i+2}$ is an SI syllable if and only if $a_{i+1}$ is even.
\bigskip
\item[(5)] If both $a_i$ and $a_{i+1}$ are NSI syllables and $a_{i+1}$ is odd, then $a_{i+2}$ is an SI syllable.
\bigskip
\item[(6)] If both $a_i$ and $a_{i+1}$ are NSI syllables and $a_{i+1}$ is even, then $a_{i+2}$ is an NSI syllable. 
\end{itemize}
\end{proposition}

\begin{proof} Condition~(1) is true by Definition~\ref{tangleNSI} because the crossings of the first syllable $a_1$ are necessarily between the two strands of the rational tangle. 
The proof of the forward direction of Condition~(2) proceeds by contraposition and follows from the left side of Figure~\ref{SI2combined}. The proof of the reverse direction of Condition~(2) follows from the right side of Figure~\ref{SI2combined}. 

\begin{figure}
\centering
\includegraphics[width=5.25in]{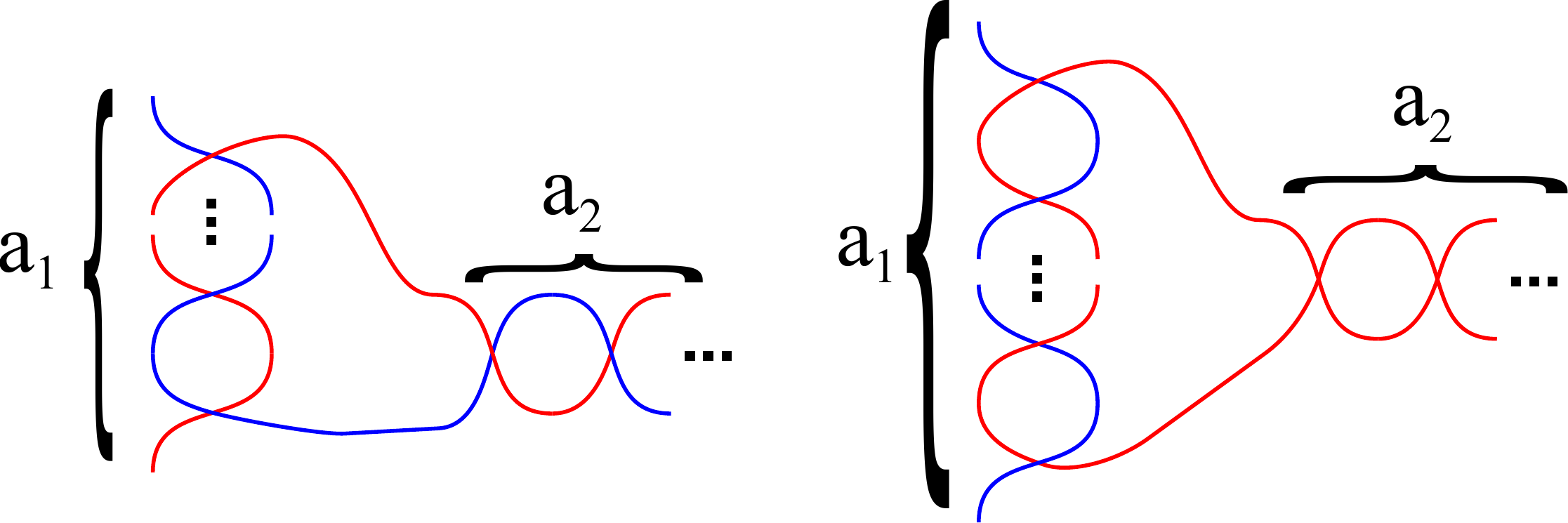}
\caption{The case where $a_1$ is odd (left) and the case where $a_1$ is even (right).}
\label{SI2combined}
\end{figure}



For the remainder of this proof, let $L$ denote the rational tangle $(a_1,a_2,\ldots,a_{i-1})$ that precedes the syllable $a_i$. Note that when we construct a rational tangle, the northwest endpoint of $L$ remains fixed. Let us assume that the strand of $L$ incident to this endpoint is colored blue and assume that the other strand of $L$ is colored red. Furthermore, we can assume that $a_i$ is a horizontal twist without loss of generality, as we can create the case where $a_i$ is a vertical twist by reflecting the diagram over the line $y=-x$.

To prove Condition~(3), assume $a_i$ is an SI syllable. Then, as shown in Figure~\ref{SI3}, the next syllable $a_{i+1}$ will be an NSI syllable. 

For the remaining conditions (Condition~(4), Condition~(5), and Condition~(6)), we will consider cases. First, we have three cases depending on the location of the second endpoint of the blue strand in $L$. Second, for each case, we have subcases that arise from considering the parities of $a_i$ and $a_{i+1}$. Figure~\ref{BSW}, Figure~\ref{BNE}, and Figure~\ref{BSE} show these 10 total cases, grouped by the location of the second endpoint of the blue strand of $L$.

To prove Condition~(4), assume $a_i$ is an SI syllable. The proof of the forward direction of Condition~(4) proceeds by contraposition and follows from the left side of Figure~\ref{BSW}. The proof of the reverse direction of Condition~(4) follows from the right side of Figure~\ref{BSW}. 

To prove Condition~(5) and Condition~(6) (combined), assume both $a_i$ and $a_{i+1}$ are NSI syllables. The proof of these conditions follows from the top two subfigures of Figure~\ref{BNE} and the bottom two subfigures of Figure~\ref{BSE}. 

\begin{figure}
\centering
\includegraphics[width=1.75in]{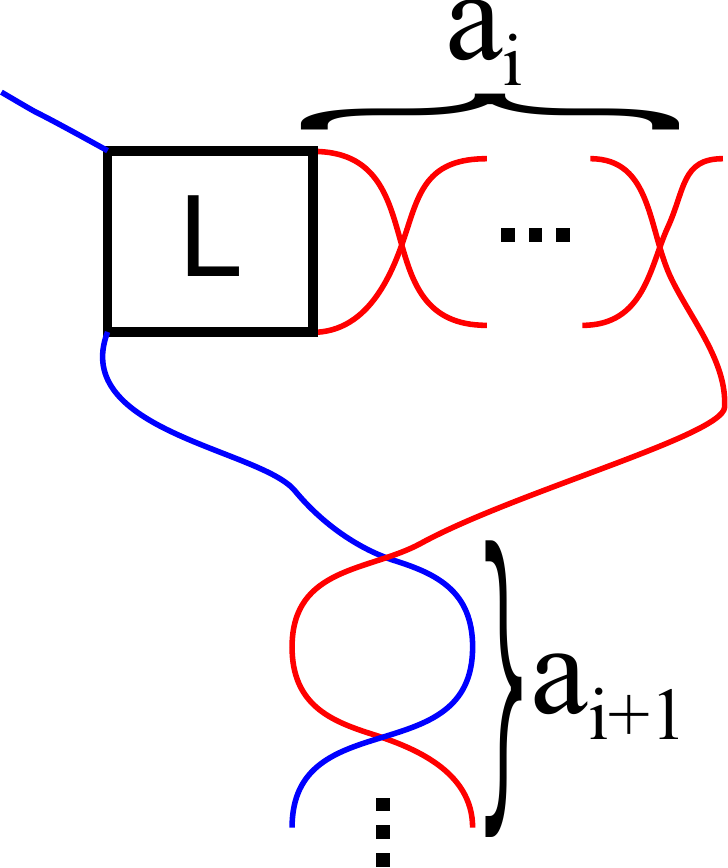}
\caption{A self-intersection syllable followed by a non-self-intersection syllable.}
\label{SI3}
\end{figure}

\begin{figure}
\centering
\includegraphics[width=6in]{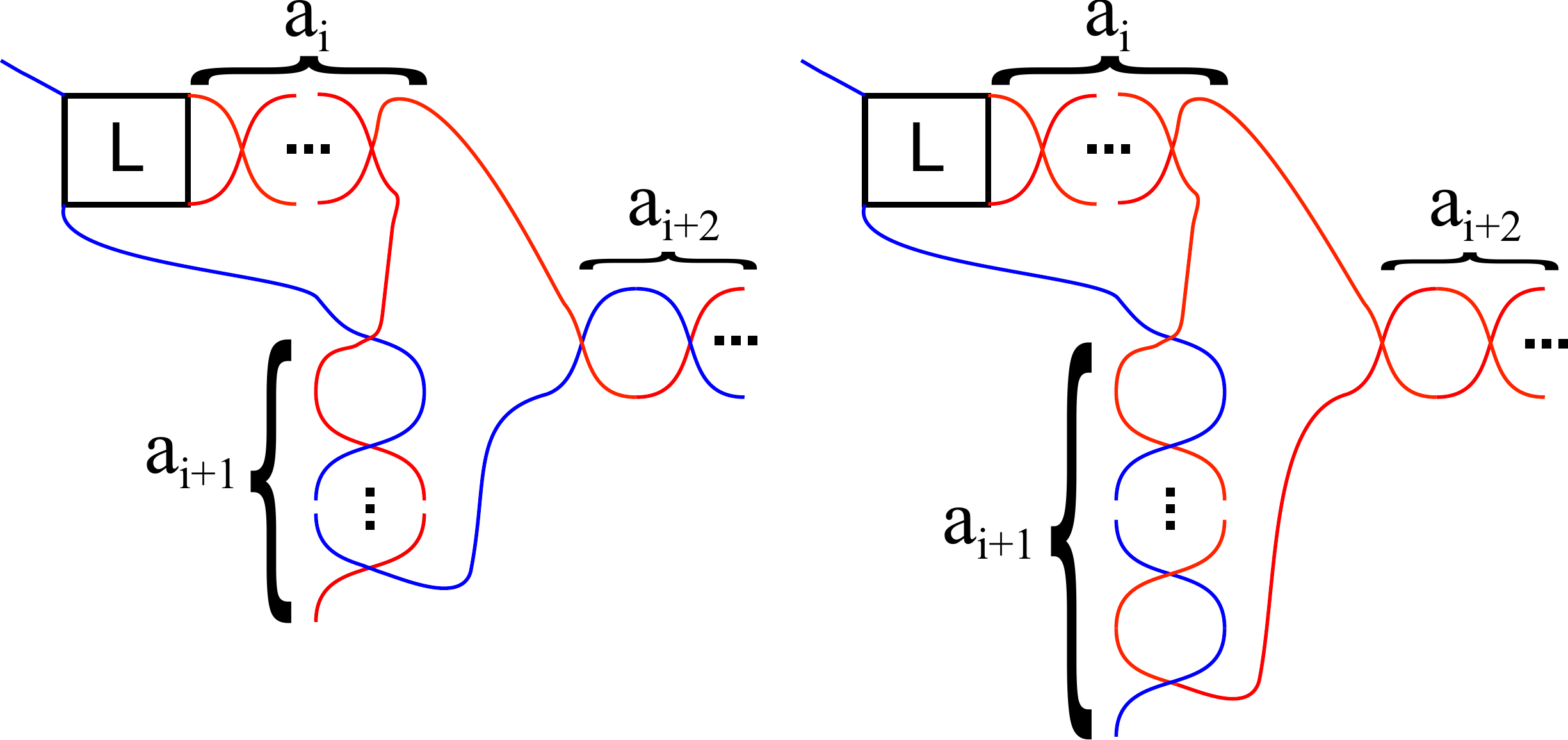}
\caption{The two cases when the southwest endpoint of $L$ is blue.}
\label{BSW}
\end{figure}

\begin{figure}
\centering
\includegraphics[width=6in]{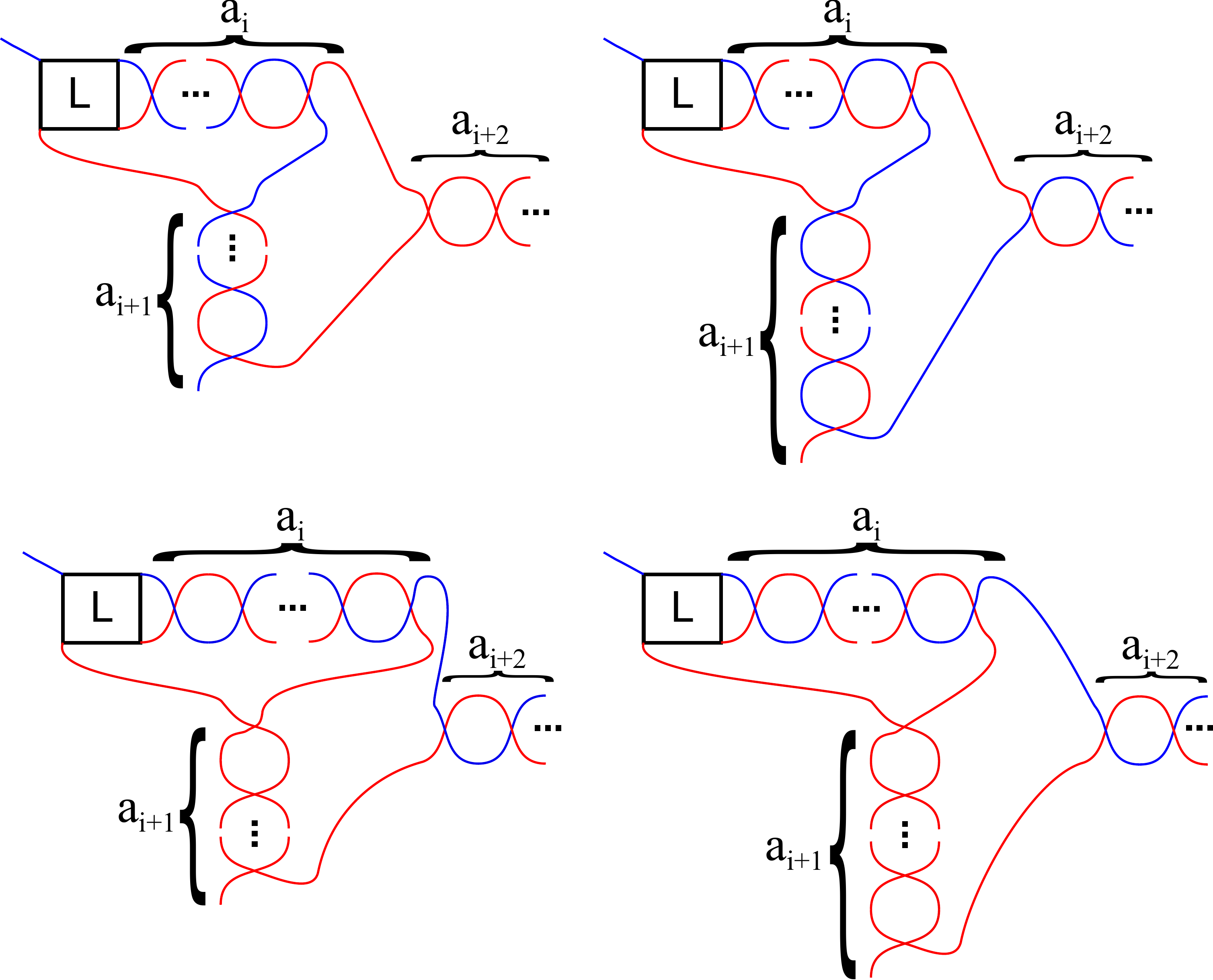}
\caption{The four cases when the northeast endpoint of $L$ is blue.}
\label{BNE}
\end{figure}

\begin{figure}
\centering
\includegraphics[width=6in]{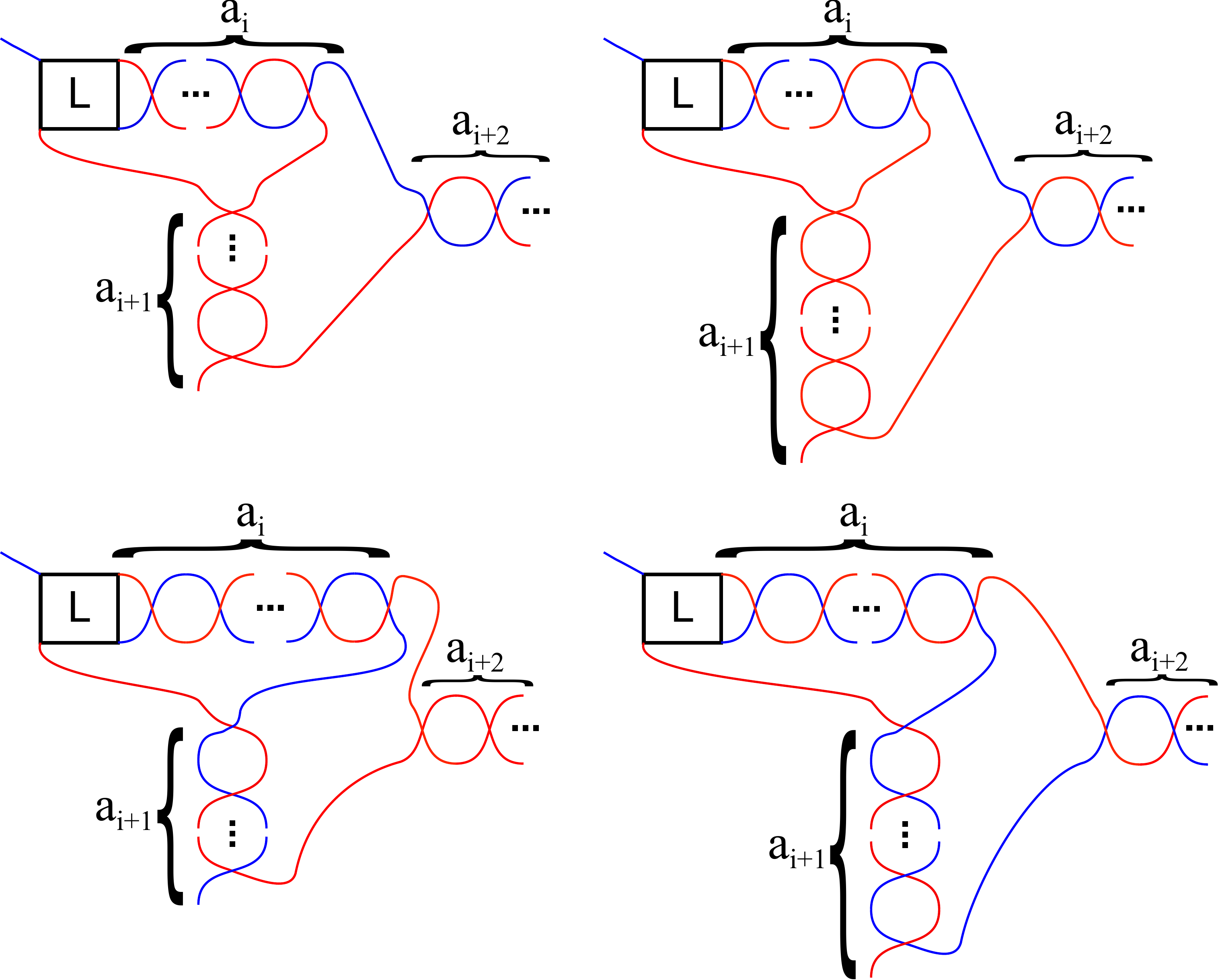}
\caption{The four cases when the southeast endpoint of $L$ is blue.}
\label{BSE}
\end{figure}
\end{proof}

We now present notation used to highlight the presence of isolated SI syllables in rational tangle words. 

\begin{notation}
Let $a$ denote a syllable of a rational tangle word. We use an asterisk $^{*}$ to indicate that the syllable consists of SIs. Thus, a syllable denoted by $a^{*}$ means that the syllable $a$ is an SI syllable.
\end{notation}

Given Proposition~\ref{SI}, we now present a result that leads to a method to decompose any rational pseudotangle word into SI syllables and strings of NSI syllables. 

\begin{proposition} \label{NSIs}
Every rational tangle word can be decomposed into strings of NSI syllables that alternate with isolated SI syllables. Furthermore, 
\begin{enumerate}
\item[(1)] all but the last string of NSI syllables consists of 
\begin{enumerate}
\item[(a)] a single even syllable, 
\item[(b)] two consecutive odd syllables, or 
\item[(c)] an odd syllable followed by an arbitrary nonempty string of even syllables followed by a final odd syllable. 
\end{enumerate}
\item[(2)] if the rational tangle contains an even number of NSIs, then the last string of NSI syllables consists of 
\begin{enumerate}
\item[(a)] a single even syllable, 
\item[(b)] two consecutive odd syllables, or 
\item[(c)] an odd syllable followed by an arbitrary nonempty string of even syllables followed by a final odd syllable.
\end{enumerate}
\item[(3)] if the rational tangle contains an odd number of NSIs, then the last string of NSI syllables consists of 
\begin{enumerate}
\item[(a)] a single odd syllable, or 
\item[(b)] an odd syllable followed by an arbitrary nonempty string of even syllables.
\end{enumerate}
\end{enumerate}
A similar statement can also be made for rational pseudotangle words.
\end{proposition}

\begin{proof} 
By Condition (3) of Proposition~\ref{SI}, no two SI syllables can be adjacent. Thus, the strings of NSI syllables alternate with isolated SI syllables. 

We will now prove Condition (1). Suppose we have a non-final string of NSI syllables. We want to show that this string consists of a single even syllable or an odd syllable followed by an arbitrary (possibly empty) string of even syllables followed by a final odd syllable. In Case~1 and Case~2 below, we consider the first non-final string of NSI syllables, whose first syllable is $a_1$. By Condition~(1) of Proposition~\ref{SI}, we know that $a_1$ is an NSI syllable. 

\bigskip

\noindent \textbf{\underline{Case 1}:} Suppose $a_1$ is even. Then Condition (2) of Proposition~\ref{SI} implies that $a_2$ is an SI syllable, so the first non-final string of NSI syllables consists of the single even syllable $a_1$.

\bigskip

\noindent \textbf{\underline{Case 2}:} Suppose $a_1$ is odd. Then Condition~(2) of Proposition~\ref{SI} implies that $a_{2}$ is an NSI syllable. If $a_{2}$ is odd, then Condition (5) of Proposition~\ref{SI} implies that $a_3$ is an SI syllable, which gives a string of NSI syllables composed of two odd syllables. If $a_2$ is even, then Condition~(6) of Proposition~\ref{SI} implies that $a_3$ is an NSI syllable. This string of even NSI syllables continues by repeatedly applying Condition~(6) of Proposition~\ref{SI}. This string of even NSI syllables must eventually terminate, however, since this is a non-final string of NSI syllables. Therefore, we eventually find an odd syllable in this string of NSI syllables, call this syllable $a_k$. This syllable $a_k$ is the last syllable in this string of NSI syllables because the next syllable $a_{k+1}$ is an SI syllable by Condition~(5) of Proposition~\ref{SI}. Thus, the first non-final string of NSI syllables consists of an odd syllable $a_1$ followed by an arbitrary (possibly empty) string $a_2, \ldots, a_{k-1}$ of even syllables followed by a final odd syllable $a_k$.

\bigskip

In Case~3 and Case~4 below, we consider a non-first non-final string of NSI syllables. Let $a_i$ denote the first syllable in this string of NSI syllables. 

\bigskip

\noindent \textbf{\underline{Case 3}:} Suppose $i>1$ and $a_i$ is even. Since $a_i$ is the first NSI syllable in a non-first string of NSI syllables, we know that $a_{i-1}$ is an SI syllable. By Condition (4) of Proposition~\ref{SI}, $a_{i+1}$ is an SI syllable, so the non-first non-final string of NSI syllables consists of the single even syllable $a_i$.

\bigskip

\noindent \textbf{\underline{Case 4}:} Suppose $i>1$ and $a_i$ is odd. Since $a_i$ is the first NSI syllable in a non-first string of NSI syllables, we know that $a_{i-1}$ is an SI syllable. Then Condition (4) of Proposition~\ref{SI} implies that $a_{i+1}$ is an NSI syllable. The remainder of this case is similar to Case 2, except that $a_{i+1}$ is now playing the role of $a_2$. 

\bigskip

This completes the proof of Condition~(1). To prove Condition~(2) and Condition~(3), we will consider the final string of NSI syllables. Let $a_i$ be the first syllable in this final string of NSI syllables. We will consider two cases based on the parity of the NSIs.

\bigskip

To prove Condition~(2), suppose there are an even number of NSIs in the tangle word. Condition~(1) of this proposition implies that each non-final string of NSI syllables contains an even number of NSIs. Thus, the final string of NSI syllables must contain an even number of NSIs for the tangle word to contain an even total number of NSIs. 

\bigskip

\noindent \textbf{\underline{Case A}:} Suppose $a_i$ is even. If the final string of NSI syllables consists of the syllable $a_i$ alone, then we have the desired result. We may now assume that the final string of NSI syllables contains a second syllable $a_{i+1}$. If the final string of NSI syllables is the only string of NSI syllables in the tangle word, then $a_i=a_1$, $a_{i+1}=a_2$, and the argument from Case~1 gives the desired result. Now suppose there are at least two strings of NSI syllables in the tangle word. Then the argument from Case~3 gives the desired result.   

\bigskip

\noindent \textbf{\underline{Case B}:} Suppose $a_i$ is odd. If the final string of NSI syllables consists of the syllable $a_i$ alone, then we have a contradiction of the fact that the final string of NSI syllables contains an even number of NSIs. We may now assume that the final string of NSI syllables contains a second syllable $a_{i+1}$. 

\bigskip

\noindent \textbf{\underline{Subcase 1}:} Suppose the final string of NSI syllables is the only string of NSI syllables in the tangle word. Then $a_i=a_1$ and $a_{i+1}=a_2$. 

Suppose $a_2$ is odd. If the final string of NSI syllables contains only $a_1$ and $a_2$, then we have the desired result. If the final string of NSI syllables contains a third syllable $a_{3}$, then the argument from Case~2 gives the desired result. 

Suppose $a_2$ is even. If the final string of NSI syllables contains only $a_1$ and $a_2$, then we have a contradiction of the fact that the final string of NSI syllables contains an even number of NSIs. If the final string of NSI syllables contains a third syllable $a_{3}$, then the argument from Case~2 implies that $a_1$ is followed by a nonempty string of even syllables. Eventually, an odd syllable must occur because the total number of NSIs in the final string of NSI syllables must be even. Then either this odd syllable is the last syllable of the tangle word or the next syllable is an SI syllable by Condition~(5) of Proposition~\ref{SI}. In either case, we have the desired result. 

\bigskip

\noindent \textbf{\underline{Subcase 2}:} Now suppose there are at least two strings of NSI syllables in the tangle word. 

Suppose $a_{i+1}$ is odd. If the final string of NSI syllables contains only $a_i$ and $a_{i+1}$, then we have the desired result. If the final string of NSI syllables contains a third syllable $a_{i+2}$, then the argument from Case~4 gives the desired result. 

Suppose $a_{i+1}$ is even. If the final string of NSI syllables contains only $a_i$ and $a_{i+1}$, then we have a contradiction of the fact that the final string of NSI syllables contains an even number of NSIs. If the final string of NSI syllables contains a third syllable $a_{i+2}$, then the argument from Case~4 implies that $a_i$ is followed by a nonempty string of even syllables. Eventually, an odd syllable must occur because the total number of NSIs in the final string of NSI syllables must be even. Then either this odd syllable is the last syllable of the tangle word or the next syllable is an SI syllable by Condition~(5) of Proposition~\ref{SI}. In either case, we have the desired result. 

\bigskip

To prove Condition~(3), suppose there are an odd number of NSIs in the tangle word. Condition~(1) of this proposition implies that each non-final string of NSI syllables contains an even number of NSIs. Thus, the final string of NSI syllables must contain an odd number of NSIs for the tangle word to contain an odd total number of NSIs. 

\bigskip

\noindent \textbf{\underline{Case A}:} Suppose $a_i$ is even. If the final string of NSI syllables consists of the syllable $a_i$ alone, then we have a contradiction of the fact that the final string of NSI syllables contains an odd number of NSIs. We may now assume that the final string of NSI syllables contains a second syllable $a_{i+1}$. If the final string of NSI syllables is the only string of NSI syllables in the tangle word, then $a_i=a_1$, $a_{i+1}=a_2$, and the argument from Case~1 gives a contradiction of the fact that the final string of NSI syllables contains an odd number of NSIs. Now suppose there are at least two strings of NSI syllables in the tangle word. Then the argument from Case~3 gives a contradiction of the fact that the final string of NSI syllables contains an odd number of NSIs.   

\bigskip

\noindent \textbf{\underline{Case B}:} Suppose $a_i$ is odd. If the final string of NSI syllables consists of the syllable $a_i$ alone, then we have the desired result. We may now assume that the final string of NSI syllables contains a second syllable $a_{i+1}$. 

\bigskip

\noindent \textbf{\underline{Subcase 1}:} Suppose the final string of NSI syllables is the only string of NSI syllables in the tangle word. Then $a_i=a_1$ and $a_{i+1}=a_2$. 

Suppose $a_2$ is odd. If the final string of NSI syllables contains only $a_1$ and $a_2$, then we have a contradiction of the fact that the final string of NSI syllables contains an odd number of NSIs. If the final string of NSI syllables contains a third syllable $a_{3}$, then the argument from Case~2 gives a contradiction of the fact that the final string of NSI syllables contains an odd number of NSIs. 

Suppose $a_2$ is even. If the final string of NSI syllables contains only $a_1$ and $a_2$, then we have the desired result. If the final string of NSI syllables contains a third syllable $a_{3}$, then the argument from Case~2 implies that $a_1$ is followed by a nonempty string of even syllables which must terminate because the tangle word is finite. This gives the desired result. 

\bigskip

\noindent \textbf{\underline{Subcase 2}:} Now suppose there are at least two strings of NSI syllables in the tangle word. 

Suppose $a_{i+1}$ is odd. If the final string of NSI syllables contains only $a_i$ and $a_{i+1}$, then we have a contradiction of the fact that the final string of NSI syllables contains an odd number of NSIs. If the final string of NSI syllables contains a third syllable $a_{i+2}$, then the argument from Case~4 gives a contradiction of the fact that the final string of NSI syllables contains an odd number of NSIs. 

Suppose $a_{i+1}$ is even. If the final string of NSI syllables contains only $a_i$ and $a_{i+1}$, then we have the desired result. If the final string of NSI syllables contains a third syllable $a_{i+2}$, then the argument from Case~4 implies that $a_i$ is followed by a nonempty string of even syllables which must terminate because the tangle word is finite. This gives the desired result. 
\end{proof}

We now present an example of applying Proposition~\ref{NSIs} to decompose a rational tangle word into SI syllables and NSI syllables. 

\begin{example} Consider the rational tangle word $$(1,4,2,1,3,5,3,2,1,2,0,5,2,6,4).$$ Reading from left to right, we see that our first string of NSI syllables is $1,4,2,1$, as this is an odd syllable followed by a nonempty string of even syllables followed by a final odd syllable. The next syllable, 3, is an SI syllable. We now update our tangle word to get $$(1,4,2,1,3^*,5,3,2,1,2,0,5,2,6,4).$$ Proceeding to the right, we see that $5,3$ is the next string of NSI syllables, as this is a string of two consecutive odd syllables. The next syllable, 2, is an SI syllable, so we update our tangle word to get $$(1,4,2,1,3^*,5,3,2^*,1,2,0,5,2,6,4).$$ Continuing this process, we see that $1,2,0,5$ is the next string of NSI syllables and the following syllable, 2, is an SI syllable. Updating our tangle word, we get $$(1,4,2,1,3^*,5,3,2^*,1,2,0,5,2^*,6,4).$$ The final string of NSI syllables consists of the single even syllable 6, so our rational tangle word ends with the syllable 4, which is an SI syllable. Thus, our completely decomposed tangle word is given by
$$(1,4,2,1,3^*,5,3,2^*,1,2,0,5,2^*,6,4^*).$$ 
\end{example}

\section{The Linking-Unlinking Game}  \label{LUgame}

In this section, we will define the Linking-Unlinking Game and provide winning strategies for the Linking-Unlinking Game played on all shadows of 2-component rational tangle closures and played on a large family of general 2-component link shadows.

\subsection{Defining the Linking-Unlinking Game} \label{LUdef}

In what follows, we introduce the Linking-Unlinking Game and present two key player strategies that will be used often in the remainder of this paper.

\begin{definition} The \textbf{Linking-Unlinking Game} is a two-player game played on the shadow of a 2-component link. The game is played with each player taking turns resolving crossings of the link shadow. The game ends when all of the crossings are resolved and a 2-component link diagram is formed. One player, the \textbf{Linker}, wins if the resulting link diagram is unsplittable. The other player, the \textbf{Unlinker}, wins if the resulting link diagram is splittable.
\end{definition}

An example of the Linking-Unlinking Game is provided below. 

\begin{example} Figure~\ref{exgame} shows an example of the Linking-Unlinking Game being played on a shadow of the Whitehead link. Here, the Unlinker goes first. The second frame shows that the Unlinker decides to play on the central crossing. Next, the Linker plays on the upper left crossing, as shown in the third frame. In the fourth frame, the Unlinker responds on the lower left crossing, resolving the crossing with the intention of being able to reduce the two left crossings with an R2 move. Next, the Linker plays on the top right crossing, as shown in the fifth frame. Finally, the sixth frame shows that the Unlinker responds on the lower right crossing, wanting to be able to reduce the two right crossings with an R2 move. The link diagram at the end of the game is equivalent to the splittable trivial 2-component link diagram (also called the 2-component unlink diagram), so the Unlinker wins. Figure~\ref{ExGameFin} shows this equivalence.
\end{example}

\begin{figure}
    \centering
    \includegraphics[width=3.75in]{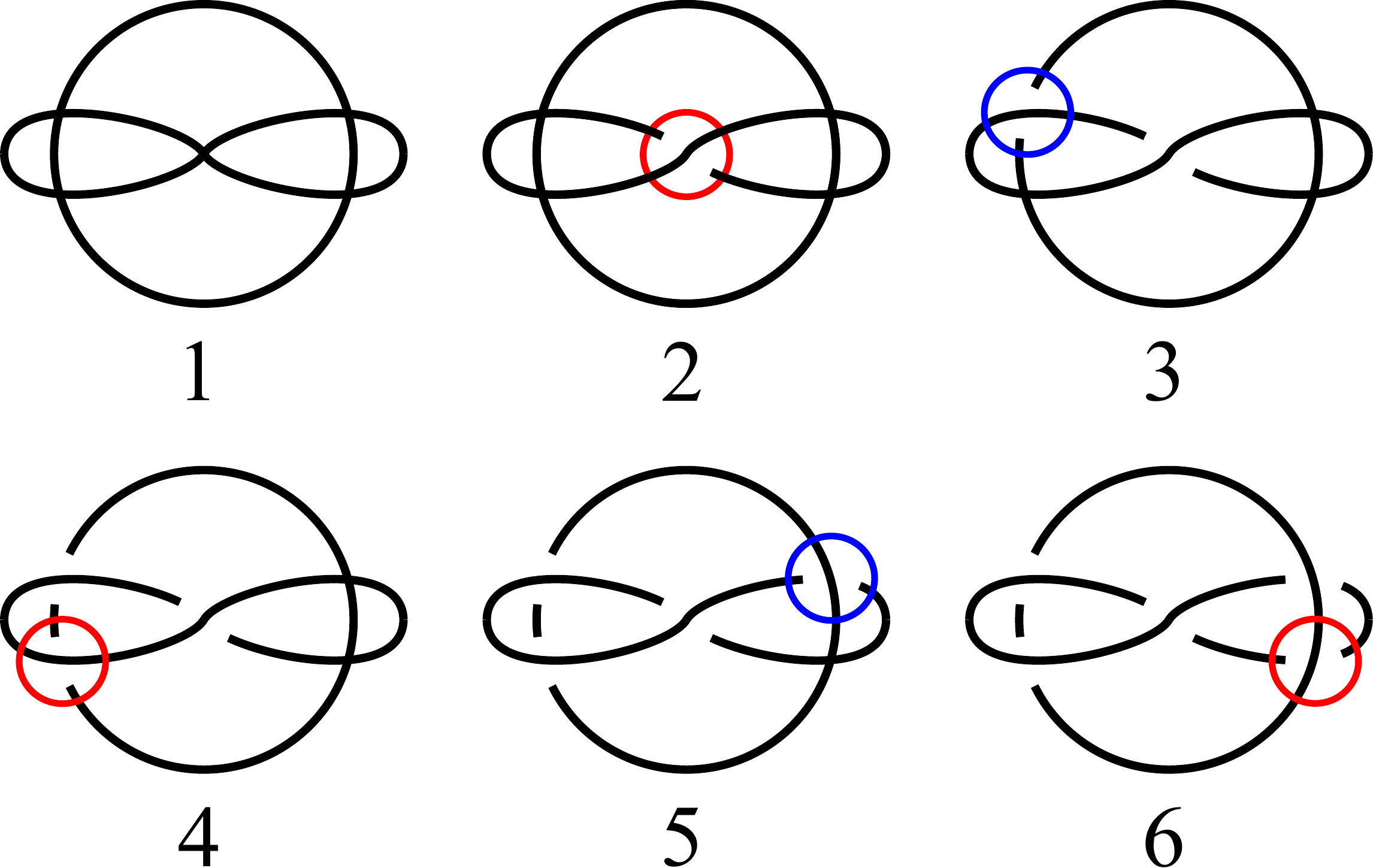}
    \caption{The Linking-Unlinking Game played on a shadow of the Whitehead link. Red circles denote Unlinker moves and blue circles denote Linker moves.}
    \label{exgame}
\end{figure}

\begin{figure}
    \centering
    \includegraphics[width=6in]{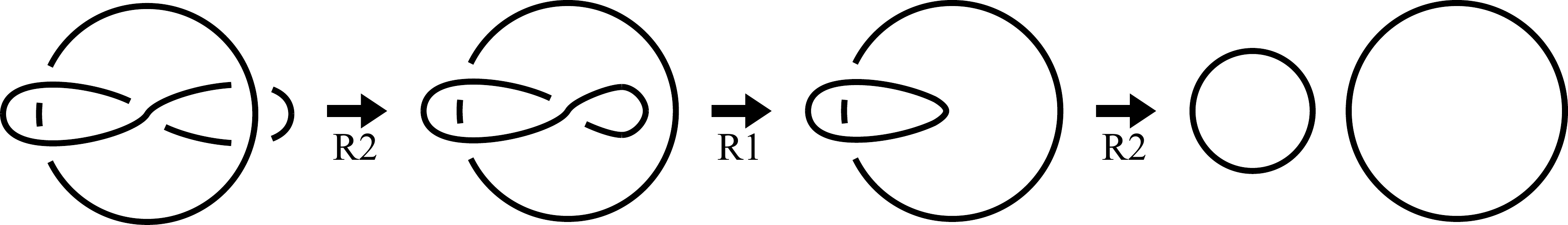}
    \caption{A sequence of Reidemeister moves used to show that the 2-component link diagram above on the left is splittable.}
    \label{ExGameFin}
\end{figure}

Given the definition of the Linking-Unlinking Game as a two-player combinatorial game, our main goal will now be to determine who wins and loses the game for various 2-component link shadows. 

\begin{definition}\label{winning}
In a two-player game, a player is said to have a \textbf{winning strategy} if there is always a sequence of moves that allows them to win, regardless of what their opponent does.
\end{definition}

Given the structure of rational pseudotangles, we will often group together crossings in a syllable of the pseudotangle word when creating winning strategies for the Linking-Unlinking Game. 

\begin{definition} \label{TR}
A \textbf{subtwist region} in a link pseudodiagram is a section of the associated link shadow consisting of a chain of bigons. A \textbf{twist region} in a link pseudodiagram is a section of the associated link shadow consisting of a maximal chain of bigons, maximal in the sense that the chain cannot be extended by adding more bigons. Note that a twist region consisting of a single crossing with no incident bigons is possible. 
\end{definition}

A chain of bigons can be thought of as being formed by twisting two parallel strands. Figure~\ref{Twist} provides a schematic depiction of a twist region. Given a subtwist region of a link pseudodiagram, we now define an operation called flyping that will be useful in defining two important player strategies for the Linking-Unlinking Game. 

\begin{figure}
    \centering
    \includegraphics[width=1.5in]{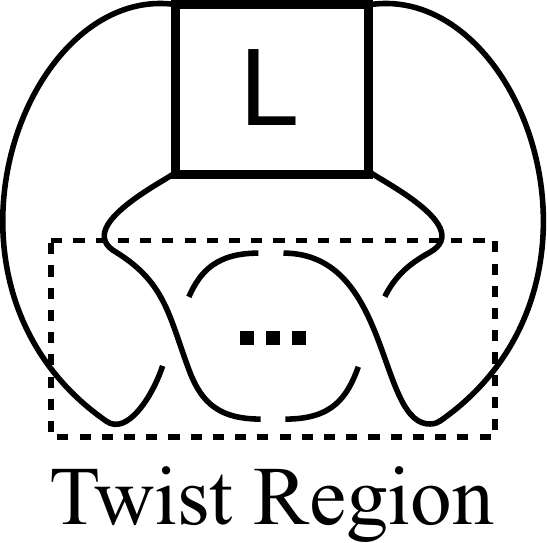}
    \caption{A schematic diagram of a twist region. The box labeled L represents the remainder of the link pseudodiagram.}
    \label{Twist}
\end{figure}

\begin{definition}
A \textbf{flype} is a move performed on a tangle of a link pseudodiagram that reflects the tangle over an axis. 
\end{definition} 

See Figure~\ref{flype} for an example of a flype. For our purposes, we will only consider flypes applied to subtwist regions of a tangle pseudodiagram. Below we present two key player strategies for the Linking-Unlinking Game that will be applied throughout the remainder of this paper. 

\begin{figure}
    \centering
    \includegraphics[width=2.25in]{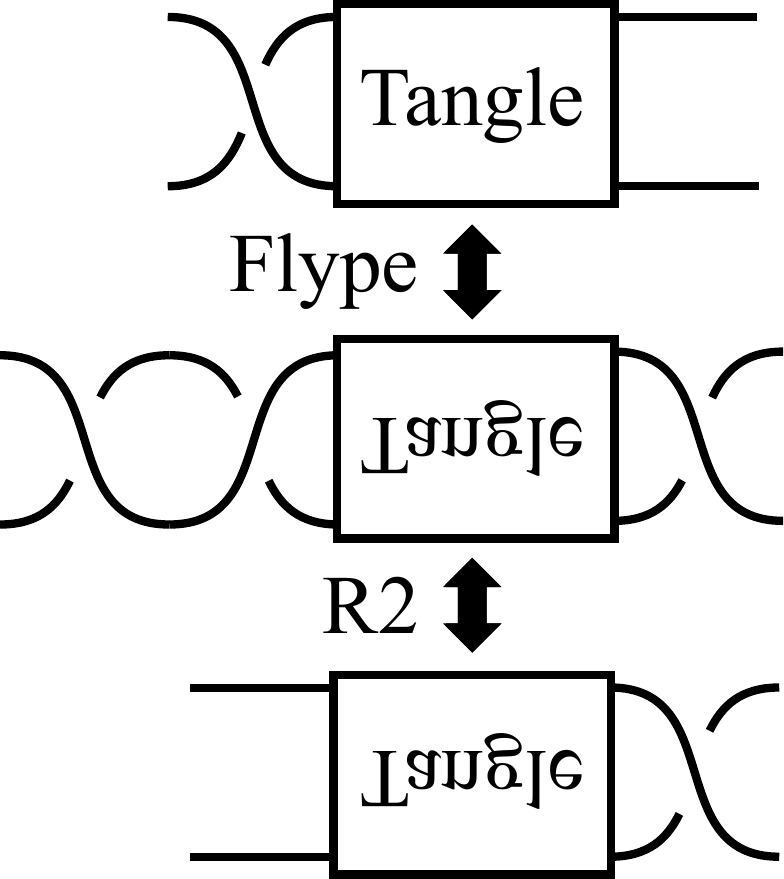}
    \caption{A schematic diagram of a flype applied to a tangle in a link pseudodiagram.}
    \label{flype}
\end{figure}

\begin{strategy}{\textbf{(R2 Strategy)}}
Assign the given link pseudodiagram an orientation. When one player plays by resolving a crossing in a twist region (consequently assigning this crossing a crossing sign), the next player will respond by resolving a crossing in the same twist region so that this crossing has crossing sign opposite to the previous resolved crossing. This allows both crossings to be reduced by applying a flype and R2 moves to the subtwist region between the two crossings, as shown in Figure~\ref{R2Strat}.
\end{strategy}

\begin{figure}
    \centering
    \includegraphics[width=3in]{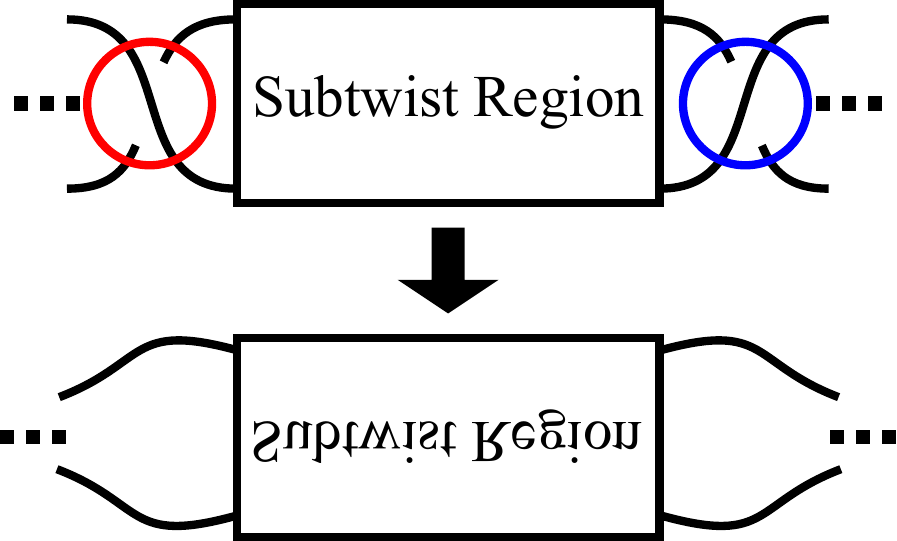}
    \caption{An example of the R2 strategy. One player resolves the crossing in the red circle and their opponent responds on the crossing in the blue circle. Applying a flype to the subtwist region between the two resolved crossings will produce two pairs of adjacent resolved crossings that can be removed by applying R2 moves.}
    \label{R2Strat}
\end{figure}

The R2 strategy is useful for both players, as it allows for the link diagram to be simplified by R2 moves at the end of the game. 

\begin{strategy}{\textbf{(Anti-R2 Strategy)}} Assign the given link pseudodiagram an orientation. Here, when one player plays by resolving a crossing in a twist region, the next player will respond by resolving a crossing in the same twist region so that this crossing has the same crossing sign as the previous resolved crossing. In this situation, the two crossings cannot be reduced by applying a flype to the subtwist region between the two crossings, as shown in Figure~\ref{AR2Strat}. Note that one or both of these crossings may be able to be reduced, however, by applying Reidemeister moves to the remainder of the link diagram at the end of the game.  
\end{strategy}

The anti-R2 strategy is particularly useful for the Linker, as it creates a clasp between the two strands of the twist region. As an example of why creating clasps might be useful for the Linker, consider the unsplittable Hopf link diagram in Figure~\ref{unsplit}. 

\begin{figure}
    \centering
    \includegraphics[width=3.5in]{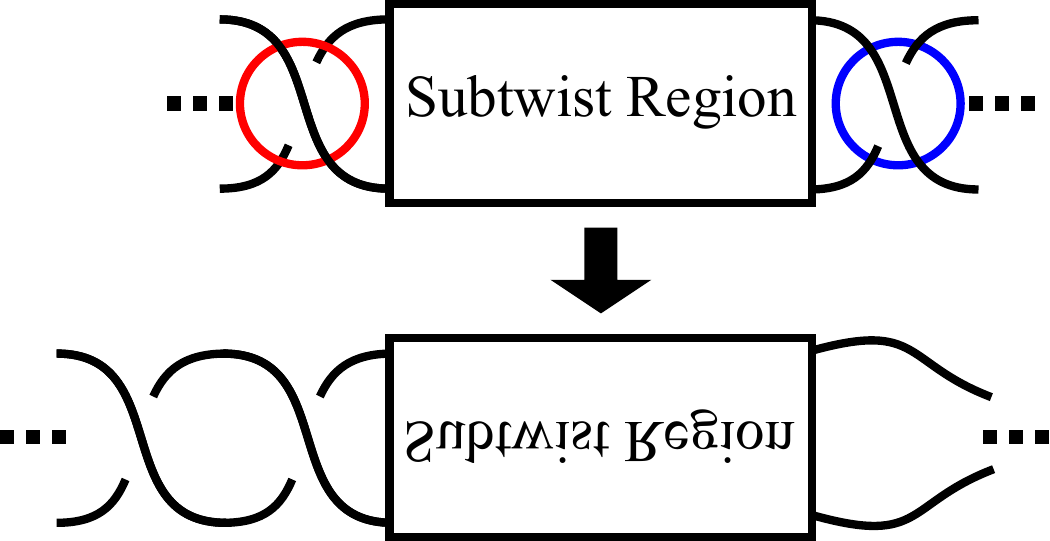}
    \caption{An example of the anti-R2 strategy. One player resolves the crossing in the red circle and their opponent responds on the crossing in the blue circle. Applying a flype and an R2 move to the subtwist region between the two resolved crossings will produce a clasp formed by two adjacent resolved crossings.}
    \label{AR2Strat}
\end{figure}

\subsection{The Linking-Unlinking Game for Rational 2-Component Link Shadows} \label{stratsratlink} 
 
In this section, we provide winning strategies for the Linking-Unlinking Game played on shadows of 2-component rational tangle closures. We begin by providing a lemma that will be used in the proofs of a number of theorems in the remainder of this paper. 

\begin{lemma} \label{Lemma}
The following statements are true. 
\begin{itemize}
\item[(1)] The rational tangle $(0)$ has a splittable 2-component link closure. 
\item[(2)] The rational tangle $(\pm 2)$ has an unsplittable 2-component link closure. 
\end{itemize}
\end{lemma}

\begin{proof}
To prove Statement~(1), observe from Figure~\ref{RT0} that the denominator closure of the tangle $(0)$ gives the trivial 2-component link diagram. This link diagram is clearly splittable. 

To prove Statement~(2), observe that the denominator closure of the tangle $(\pm 2)$ gives either the Hopf link diagram from Figure~\ref{unsplit} or its reflection. In either case, after orienting this link diagram, it can be seen that the linking number of the resulting oriented 2-component link diagram has absolute value $1$. Since this link diagram has a nonzero linking number, then it is unsplittable.
\end{proof}

We now begin our study of the Linking-Unlinking Game by focusing on certain families of rational 2-component link shadows. The following result presents winning strategies for the Unlinker in a very specific pathological case and for the second player in other cases. 

\begin{theorem} \label{Thm1}
Suppose we have a shadow of a rational two-component link coming from a closure of the rational tangle $(a_1,\ldots,a_n)$.
\begin{enumerate}
\item[(1)] If $a_{2k+1} = 0$ for all $k$, then the Unlinker wins. 
\item[(2)] If $a_{2k+1}\neq 0$ for at least one $k$ and all of the $a_i$ are even, then the second player has a winning strategy.
\end{enumerate}
\end{theorem}

\begin{proof} We begin by proving Statement~(1). Assume we have a shadow of a rational 2-component link coming from a closure of the rational tangle $(a_1,\ldots,a_n)$ where $a_{2k+1} = 0$ for all $k$. If $n=1$, then the tangle word is $(0)$. Now suppose that $n \geq 2$. Then the tangle word can be written in the form $(0, a_2, \ldots, 0, a_{2m})$ or $(0, a_2, \ldots, 0, a_{2m}, 0)$. If $m=1$, then Statement~3 of Proposition~\ref{tangle eq} implies that $a_2$ can be reduced to $0$ so that the tangle word is equivalent to $(0, 0)$ or $(0, 0, 0)$, both of which are equivalent to $(0)$ by applying Statement~0 of Proposition~\ref{tangle eq} either once or twice, respectively. If $m \geq 2$, then we can iteratively use Statement~3 followed by Statement~2 of Proposition~\ref{tangle eq} to reduce the tangle subwords $(0, a_{2k})$ to $(0, 0)$ and remove them from the tangle word. Eventually, what remains is either $(0, 0)$ or $(0, 0, 0)$, both of which are equivalent to $(0)$ by applying Statement~0 of Proposition~\ref{tangle eq} either once or twice, respectively. By Lemma~\ref{Lemma}, the Unlinker wins. 

We will now prove Statement~(2). Assume we have a shadow of a rational 2-component link coming from a closure of the rational tangle $(a_1,\ldots,a_n)$ where $a_{2k+1}\neq 0$ for at least one $k$ and where all of the $a_i$ are even. Fix a value $j$ such that $a_{2j+1} \neq 0$. We consider two cases, depending on the role of the second player. 

\bigskip

\noindent \textbf{\underline{Case 1}:} Suppose the Unlinker plays second. Whenever the Linker plays on a syllable, the Unlinker should respond on the same syllable by using the R2 strategy. Note that this response strategy is always possible since all of the $a_{i}$ are even. Thus, each time the Unlinker responds, two crossings can be reduced. This means that the final link diagram can be reduced to form the trivial 2-component link diagram, which is clearly splittable. Therefore, the Unlinker wins. 

\bigskip

\noindent \textbf{\underline{Case 2}:} Suppose the Linker plays second. Whenever the Unlinker plays on any $a_i$ syllable where $i\neq 2j+1$, the Linker should respond on this syllable by using the R2 strategy. Again, this strategy can be applied since each $a_i$ is even. Thus, each time the Linker responds, two crossings can be reduced. When the Unlinker plays on $a_{2j+1} \neq 0$, the Linker should respond by using the R2 strategy (which allows for two crossings to be reduced) until only two unresolved crossings remain in this subtwist region. On the last two unresolved crossings in this subtwist region, the Linker should respond by using the anti-R2 strategy.

After applying flypes and R2 moves at the end of the game, the resulting tangle word will be of the form $(0,\ldots,0,2,0,\ldots,0)$ or $(0,\ldots,0,-2,0,\ldots,0)$, where there are an even number of zeros before the 2 or $-2$. By repeatedly applying Statement~2 of Proposition~\ref{tangle eq}, we can reduce the tangle word to $(2,0,\ldots,0)$ or $(-2,0,\ldots,0)$. By repeatedly applying Statement~0 of Proposition~\ref{tangle eq}, these tangle words are equivalent to $(2)$ and $(-2)$, respectively. By Lemma~\ref{Lemma}, the Linker wins.
\end{proof}

We now present winning strategies for rational 2-component link shadows whose tangle word consists of two odd syllables. 
\begin{theorem} \label{Thm2}
Playing on the shadow of a rational 2-component link coming from a closure of the rational tangle $(a_1,a_2)$, if $a_1$ and $a_2$ are both odd, then the second player has a winning strategy.
\end{theorem}

\begin{proof} Assume we have a shadow of a rational 2-component link coming from a closure of the rational tangle $(a_1,a_2)$ where $a_1$ and $a_2$ are both odd. We will consider two cases, depending on the role of the second player.

\bigskip

\noindent \textbf{\underline{Case 1}:} Suppose the Linker plays second. When the Unlinker plays, the Linker should respond on the same syllable using the R2 strategy so long as this is possible. Since the Unlinker is playing first and both of the syllables are odd, then the Unlinker will necessarily have to resolve the last crossing in either $a_1$ or $a_2$. Once the Unlinker fully resolves an $a_i$, the pseudotangle word can be reduced to one of $(1,(a)),(-1,(a)),((a),1)$, or $((a),-1)$, where $a$ is odd. The Linker will then play on $a$ by resolving a crossing to have the same overcrossing slope as the previous crossing resolved by the Unlinker. Thus, the Linker will turn $(1,(a))$ into $(1,1(|a|-1))$, $(-1,(a))$ into $(-1,-1(|a|-1))$, $((a),1)$ into $(1(|a|-1),1)$, and $((a),-1)$ into $(-1(|a|-1),-1)$.  

Since $|a|-1$ is even, then the Linker should resume using the R2 strategy to respond to the Unlinker until the end of the game. This results in a rational tangle word that can be reduced to either $(1,1)$ or $(-1,-1)$, which are equivalent to $(2)$ and $(-2)$ by Statement~4 and Statement~5 of Proposition~\ref{tangle eq}, respectively. By Lemma~\ref{Lemma}, the Linker wins. 

\bigskip

\noindent \textbf{\underline{Case 2}:} Suppose the Unlinker plays second. The Unlinker's strategy will be similar to the Linker's strategy from Case 1. First, the Unlinker should use the R2 strategy until the Linker fully resolves one of the twist regions. As in the previous case, the pseudotangle word can be reduced to the form $(1,(a)),(-1,(a)),((a),1)$, or $((a),-1)$, where $a$ is odd. The Unlinker should then play on $a$ by resolving a crossing to have overcrossing slope opposite to the previous crossing resolved by the Linker. Thus, the Unlinker will turn $(1,(a))$ into $(1,-1(|a|-1))$, $(-1,(a))$ into $(-1,1(|a|-1))$, $((a),1)$ into $(-1(|a|-1),1)$, and $((a),-1)$ into $(1(|a|-1),-1)$. 

Since $|a|-1$ is even, then the Unlinker should resume using the R2 strategy to respond to the Linker until the end of the game. This results in a rational tangle word that can be reduced to either $(1,-1)$ or $(-1,1)$, both of which are equivalent to $(0)$ by Statement~4 and Statement~5 of Proposition~\ref{tangle eq}, respectively. By Lemma~\ref{Lemma}, the Unlinker wins. 
\end{proof}

The following theorem builds on the results of the previous two theorems. 

\begin{theorem} \label{Thm3}
Playing on the shadow of a rational 2-component link coming from a closure of the rational tangle $(a_1,\ldots,a_n)$ where $n\geq3$, if $a_1$ and $a_n$ are odd and all other $a_i$ are even, then the second player has a winning strategy.
\end{theorem}

\begin{proof} Assume we have a shadow of a rational 2-component link coming from a closure of the rational tangle $(a_1,\ldots,a_n)$ where $n\geq3$, where $a_1$ and $a_n$ are odd, and where all other $a_i$ are even. When the first player plays on $a_1$ or $a_n$, the second player should respond by using the same strategy as used in the proof of Theorem~\ref{Thm2}, as though the game were played on $(a_1,a_n)$, where the Linker wants to reduce $(a_1,a_n)$ to $(1, 1)$ or $(-1, -1)$ and the Unlinker wants to reduce $(a_1,a_n)$ to $(1, -1)$ or $(-1, 1)$. When the first player plays on one of the even $a_i$ syllables, the second player should respond on the same syllable by using the R2 strategy. At the end of the game, there are two cases for the resulting tangle word, depending on the role of the second player.

\bigskip

\noindent \textbf{\underline{Case 1}:} Suppose the Linker plays second. Then the resulting tangle word can be reduced to $(1,0,\ldots,0,1)$ or $(-1,0,\ldots,0,-1)$. If there are an even number of zeros, then the tangle word can be reduced to $(1,1)$ or $(-1,-1)$ by repeatedly applying Statement~2 of Proposition~\ref{tangle eq} and reduced to $(2)$ or $(-2)$ by Statement~4 and Statement~5 of Proposition~\ref{tangle eq}, respectively. If there are an odd number of zeros, then the tangle word can be reduced to $(1,0,1)$ or $(-1,0,-1)$ by repeatedly applying Statement~2 of Proposition~\ref{tangle eq} and reduced to $(2)$ or $(-2)$ by Statement~1 of Proposition~\ref{tangle eq}. By Lemma~\ref{Lemma}, the Linker wins. 

\bigskip

\noindent \textbf{\underline{Case 2}:} Suppose the Unlinker plays second. Then the resulting tangle word can be reduced to $(1,0,\ldots,0,-1)$ or $(-1,0,\ldots,0,1)$. If there are an even number of zeros, then the tangle word can be reduced to $(1,-1)$ or $(-1,1)$ by repeatedly applying Statement~2 of Proposition~\ref{tangle eq} and reduced to $(0)$ by Statement~4 and Statement~5 of Proposition~\ref{tangle eq}, respectively. If there are an odd number of zeros, then the tangle word can be reduced to $(1,0,-1)$ or $(-1,0,1)$ by repeatedly applying Statement~2 of Proposition~\ref{tangle eq} and reduced to $(0)$ by Statement~1 of Proposition~\ref{tangle eq}. By Lemma~\ref{Lemma}, the Unlinker wins. 
\end{proof}


We will now explore how the parity of the number of SIs in the rational 2-component link shadow affects which player has a winning strategy. 

\begin{theorem} \label{Thm4}
Suppose we have a shadow of a rational 2-component link coming from a closure of a rational tangle. 
\begin{enumerate}
\item[(1)] If the tangle word contains an even number of SIs, then the second player has a winning strategy.
\item[(2)] If the tangle word contains an odd number of SIs, then the first player has a winning strategy.
\end{enumerate} 
\end{theorem}

\begin{proof} Assume we have a shadow of a rational 2-component link coming from a closure of the rational tangle $(a_1,\ldots,a_n)$. 

We begin by proving Statement~(1). Assume $(a_1,\ldots,a_n)$ contains an even number of SIs. Using Proposition~\ref{NSIs}, we can decompose $(a_1,\ldots,a_n)$ into alternating strings of NSI syllables and isolated SI syllables. Whenever the first player plays on an SI, the second player should respond by playing arbitrarily on any other SI. Note that this response strategy is always possible since the number of SIs is even.

By Proposition~\ref{ratevenNSI}, since we have a 2-component link shadow coming from a closure of a rational tangle, then the tangle must contain an even number of NSIs. Therefore, by Proposition~\ref{NSIs}, we have that a non-final or final string of NSI syllables can be either (1) a single even syllable, (2) two consecutive odd syllables, or (3) an odd syllable followed by an arbitrary nonempty string of even syllables followed by a final odd syllable. Notice that the proofs of Statement~(2) of Theorem~\ref{Thm1}, the statement of Theorem~\ref{Thm2}, and the statement of Theorem~\ref{Thm3} provide the second player with a strategy for playing on the respective type~(1), type~(2), and type~(3) strings of NSI syllables listed above. 

When the first player plays on an unresolved crossing from a non-final string of NSI syllables, the second player should respond on an unresolved crossing in the same non-final string of NSI syllables by using the strategy for when the Unlinker plays second provided by the proof of the theorem corresponding to the type of string of NSI syllables. This strategy ensures that each non-final string of NSI syllables can be reduced to the tangle subword $(0)$. 

When the first player plays on an unresolved crossing from the final string of NSI syllables, the second player should respond on an unresolved crossing in the same final string of NSI syllables by using the strategy that corresponds to their role as Linker or Unlinker from the proof of the theorem corresponding to the type of string of NSI syllables. 

Note that the second player will always be able to respond on the same non-final or final string of NSI syllables because each such string contains an even total number of crossings. We will now consider two cases, depending on the role of the second player.

\bigskip

\noindent \textbf{\underline{Case 1}:} Suppose the Unlinker plays second. Let $(a_1,\ldots,a_k)$ denote the first string of NSI syllables. By the proofs of Statement~(2) of Theorem~\ref{Thm1}, the statement of Theorem~\ref{Thm2}, and the statement of Theorem~\ref{Thm3}, we know that the Unlinker's strategy results in this string of syllables being able to be reduced to the tangle subword $(0)$. 

If the entire tangle word consists of a single string of NSI syllables, then the tangle word $(0)$ closes to form the trivial 2-component link diagram and the Unlinker wins. If the entire tangle word consists of two or more strings of NSI syllables, then the tangle word can be reduced to $(0,a_{k+1}^*,a_{k+2},\ldots,a_n)$. By Statement~3 of Proposition~\ref{tangle eq}, the syllable $a_{k+1}$ can be reduced to $0$. As a result, the entire tangle word reduces to $(0,0,a_{k+2},\ldots,a_n)$, which reduces to $(a_{k+2},\ldots,a_n)$ by Statement~2 of Proposition~\ref{tangle eq}. Since $a_{k+2}$ is the start of a string of NSI syllables, then we can iterate this process until we are left with a tangle word that is equivalent to $(0)$. By Lemma~\ref{Lemma}, the Unlinker wins.

\bigskip

\noindent \textbf{\underline{Case 2}:} Suppose the Linker plays second. Then, as indicated in the paragraph preceding Case~1, the only difference in strategy will occur in the final string of NSI syllables. Specifically, the Linker (and the Unlinker) will respond so that all of the non-final strings of NSI syllables can be reduced to the tangle subword $(0)$ and the Linker will respond so that the final string of NSI syllables can be reduced to either the tangle subword $(2)$ or the tangle subword $(-2)$. 

If there is no isolated SI syllable following this final string of NSI syllables, then the tangle word can be reduced to $(2)$ or $(-2)$ by iteratively using Statement~3 followed by Statement~2 of Proposition~\ref{tangle eq} to reduce the tangle subwords $(0, a_{2k})$ coming from non-final strings of NSI syllables followed by an SI syllable to $(0, 0)$ and remove them from the tangle word. By Lemma~\ref{Lemma}, the Linker wins. If there is an isolated SI syllable, call it $a^{*}$, following this final string of NSI syllables, then the tangle word can be reduced to $(2,a^{*})$ or $(-2,a^{*})$. In this case, a linking number argument similar to the one used in the proof of Lemma~\ref{Lemma} can be used to show that the resulting 2-component link diagram is unsplittable. Note that this is because the syllable $a^{*}$ consists of SIs that are not considered when computing the linking number. Therefore, the Linker wins. 

\bigskip

We will now prove Statement~(2). Assume $(a_1,\ldots,a_n)$ contains an odd number of SIs. Then the first player should begin by playing arbitrarily on an SI. This will effectively reduce the game to one with an even number of unresolved SIs where the first player is now playing second on this new link pseudodiagram. But then the first player can use the strategy from the proof of Statement~(1) to guarantee a win. 
\end{proof}

\subsection{The Linking-Unlinking Game for General 2-Component Link Shadows} \label{GeneralLinks}

We now expand our focus from rational 2-component link shadows to general 2-component link shadows. In particular, we will focus the majority of our attention on finding winning strategies for the Linker. First, we introduce the notion of a partial linking number that can be applied to a link pseudodiagram throughout the Linking-Unlinking Game. 

\begin{definition} \label{pseudo linking number}
The \textbf{pseudo-linking number} of an oriented 2-component link pseudodiagram is defined to be half of the sum of the crossing signs, where the sum is taken over all of the resolved non-self-intersections (NSIs) of the link pseudodiagram. Note that an oriented 2-component link shadow has a pseudo-linking number of 0 and note that the pseudo-linking number of an oriented 2-component link diagram is the linking number of the diagram.
\end{definition}

To conclude this paper, we use the pseudo-linking number, the presence of NSIs, and the parity of the number of SIs to present winning strategies for the Linking-Unlinking Game played on a large family of general 2-component link shadows. In particular, we present winning strategies for the Unlinker in a very specific pathological case and for the Linker in other cases. 

\begin{remark}
Since linking number arguments do not seem to work in the cases where there are an even (resp. odd) number of SIs in the link shadow and where the Linker plays first (resp. second), then these cases currently remain open in terms of finding winning strategies for playing the Linking-Unlinking Game on these link shadows. 
\end{remark}

\begin{theorem} \label{thm6} 
Suppose we have a shadow of a 2-component link. 
\begin{enumerate}
\item[(1)] If the shadow contains no NSIs, then the Unlinker wins. 
\item[(2)] If the shadow contains a nonzero number of NSIs and an even number of SIs, then the Linker has a winning strategy when playing second.
\item[(3)] If the shadow contains a nonzero number of NSIs and an odd number of SIs, then the Linker has a winning strategy when playing first.
\end{enumerate} 
\end{theorem}

\begin{proof} We begin by proving Statement~(1). Since the 2-component link shadow contains no NSIs, then the game starts on a split 2-component link shadow and the Unlinker wins automatically because resolving crossings will never create an unsplittable link diagram.

We now prove Statement~(2). Assign an arbitrary orientation to the 2-component link shadow. If the Unlinker plays on an SI, the Linker should respond by also playing on an SI. Since the number of SIs is even, the Linker will always be able to respond on a remaining SI until all SIs are resolved. Note that the SIs do not affect the (pseudo-)linking number, so how the Linker responds is arbitrary.

Recall that the pseudo-linking number of a 2-component link shadow (or a 2-component link pseudodiagram with only SIs resolved) is 0. Thus, after the Unlinker plays for the first time on an NSI, the 2-component link pseudodiagram will have a positive pseudo-linking number of $\frac{1}{2}$ or a negative pseudo-linking number of $-\frac{1}{2}$. The Linker should then respond on any other NSI by resolving the crossing to have the same sign as the crossing resolved by the Unlinker on the previous move, which will change the pseudo-linking number to $1$ or $-1$. 

For the remaining NSIs, when the Unlinker plays on an NSI, the Linker should respond on a remaining NSI by resolving the crossing to have sign opposite to the crossing resolved by the Unlinker on the previous move. Since, by Proposition~\ref{evenNSI}, every 2-component link pseudodiagram contains an even number of NSIs, the Linker will always be able to respond on a remaining NSI until all NSIs are resolved. 

Notice that the Linker's response strategy preserves the pseudo-linking number at $1$ or $-1$ (after their first response on an NSI) for the remainder of the game. Consequently, when the Linker makes the final move, the resulting 2-component link diagram will have a nonzero linking number. This implies that the link diagram is unsplittable and, therefore, the Linker wins. 

We now prove Statement~(3). Assign an arbitrary orientation to the 2-component link shadow. The Linker should begin by playing arbitrarily on an SI. This will effectively reduce the game to one with an even number of unresolved SIs where the Linker is now playing second on this new link pseudodiagram. But then the Linker can use the strategy from the proof of Statement~(2) to guarantee a win. 
\end{proof}

\bibliography{ThesisBib} 

\bibliographystyle{plain}

\end{document}